\documentclass[12pt]{amsart}

\usepackage{fullpage, amsmath,  amssymb, latexsym, stmaryrd, enumerate, tikz-cd, adjustbox, colonequals, color,mathtools,mathrsfs,graphicx}
\usepackage[all]{xy}

\definecolor{hot}{RGB}{65,105,225}
\usepackage[pagebackref=true,colorlinks=true, linkcolor=hot ,  citecolor=hot, urlcolor=hot]{hyperref}

\newtheorem{theorem}{Theorem}[section]
\newtheorem{thm}[theorem]{Theorem}
\newtheorem{lemma}[theorem]{Lemma}

\newtheorem{proposition}[theorem]{Proposition}
\newtheorem{prop}[theorem]{Proposition}
\newtheorem{conj}[theorem]{Conjecture}

\newtheorem{que}[theorem]{Question}

\theoremstyle{definition}
\newtheorem{definition}[theorem]{Definition}
\newtheorem{defn}[theorem]{Definition}
\newtheorem{example}[theorem]{Example}

\newtheorem{rmk}[theorem]{Remark}

\numberwithin{equation}{section}

\def\bZ{\mathbb{Z}}
\def\bQ{\mathbb{Q}}
\def\bN{\mathbb{N}}
\def\bP{\mathbb{P}}
\def\bC{\mathbb{C}}

\def\cO{\mathcal{O}}
\def\cL{\mathcal{L}}
\def\cF{\mathcal{F}}
\def\GL{{\rm GL}}

\def\Hom{{\rm Hom}}
\def\bbQ{\bar{\bQ}}
\def\xa{\xrightarrow}
\def\ra{\rightarrow}
\def\tr{{\rm tr}}

\def\al{\alpha}
\def\bar {\overline}
\def\sC{\mathscr{C}}
\def\sD{\mathscr{D}}
\def\sM{\mathscr{M}}
\def\be{\begin{equation}}
\def\ee{\end{equation}}

\def\Spec{{\text{Spec}}}
\def\ul{\underline}
\def\ol{\overline}

\begin{document}

\title{Absolute sets of rigid local systems}

\author{Nero Budur}\address{Department of Mathematics, KU Leuven, Celestijnenlaan 200B, 3001 Leuven, Belgium, and BCAM, Mazarredo 14, 48009 Bilbao, Spain}\email{nero.budur@kuleuven.be}

\author{Leonardo A. Lerer}\address{D\'epartement de Math\'ematiques d'Orsay, Universit\'e Paris-Saclay, F-91405 Orsay, France}\email{a.leonardo.lerer@gmail.com}

\author{Haopeng Wang}\address{Department of Mathematics, KU Leuven, Celestijnenlaan 200B, 3001 Leuven, Belgium}\email{haopeng.wang@kuleuven.be}

\keywords{Absolute set, local system, logarithmic connection, rigid local system.}


\begin{abstract}
The absolute sets of local systems on a smooth complex algebraic variety are the subject of a conjecture of \cite{BW} based on an analogy with special subvarieties of Shimura varieties. An absolute set should be the higher-dimensional generalization of a local system of geometric origin. We show that the conjecture for absolute sets of simple cohomologically rigid local systems reduces to the zero-dimensional case, that is, to Simpson's conjecture that every such local system with quasi-unipotent monodromy at infinity and determinant is of geometric origin. In particular, the conjecture holds for this type of absolute sets if the variety is a curve or if the rank is two.
\end{abstract}

\maketitle
	\tableofcontents

\section{Introduction}

Let $X$ be a smooth complex algebraic variety. The Riemann-Hilbert correspondence establishes an equivalence between complex local systems on the complex manifold underlying $X$ and regular  algebraic flat connections. For a set $S$ of isomorphism classes of complex local systems on $X$
and  $\sigma\in Gal(\bC/\bQ)$, one constructs a set $S^\sigma$ of isomorphism classes of local systems on the complex manifold underlying the conjugate variety $X^\sigma$ by applying $\sigma$ to the algebraic flat connections corresponding to the elements of $S$. Then $S$ is
called {\it absolute $\bbQ$-constructible} by \cite{BW} if for every $\sigma\in Gal(\bC/\bQ)$ the set $S^\sigma$ is $\bbQ$-constructible in an appropriate sense, see Definition \ref{defBWas}. This is a somewhat different definition than the notion of absoluteness introduced by Simpson \cite{S93} in the projective case. If $S$ is a subset of the complex points of the Betti moduli space $M_B(X,r)$ parametrizing isomorphism classes of semisimple complex rank $r$ local systems on $X$, then $S$ is absolute $\bbQ$-constructible if and only if $S^\sigma$ is $\bbQ$-constructible in $M_B(X^\sigma,r)$ for every $\sigma\in Gal(\bC/\bQ)$, see Proposition \ref{lemAnlg}. 
 Similarly one can define {\it absolute $\bbQ$-closed} sets using the Zariski topology on $M_B(X^\sigma,r)$.  One has:

 \begin{conj}[Conjecture A] \label{conj1}  Let $X$ be a smooth complex algebraic variety, and $r>0$ an integer. The following hold in $M_B(X, r)$: 
 \begin{enumerate}
 \item[A1.]    	The collection of absolute $\bar{\bQ}$-constructible subsets is generated from the absolute $\bar{\bQ}$-closed subsets via finite sequences of taking unions, intersections, and complements; the Euclidean (or equivalently, the Zariski) closure of an absolute $\bar{\bQ}$-constructible set is absolute $\bar{\bQ}$-closed; and an irreducible component of an absolute $\bar{\bQ}$-closed subset is absolute $\bar{\bQ}$-closed.
    	\item[A2.] Any non-empty absolute $\bar{\bQ}$-constructible subset contains a Zariski dense subset of absolute $\bar{\bQ}$-points.
    	\item[A3.] A point is an absolute $\bar{\bQ}$-point if and only if it is a local system of geometric origin (in the sense of \cite[6.2.4]{BBD}).
    	\item[A4.] The Zariski closure of a set of absolute $\bar{\bQ}$-points is $\bbQ$-pseudo-isomorphic to an absolute $\bar{\bQ}$-constructible subset of some $M_B(X', r')$, for possibly a different smooth complex algebraic variety $X'$  and rank $r'$. Moreover, the $\bbQ$-pseudo-isomorphism restricts to a bijection between the sets of absolute $\bbQ$-points. (See Section \ref{secnot} for the definition of $\bbQ$-pseudo-morphisms.)\footnote{Added after the editorial decision: In the recent article ``Geometric local systems on very general curves and isomonodromy" by Landesman and Litt, one of their results claims that  A2 and  A3 cannot be simultaneously true for the full Betti moduli on certain smooth projective curves of genus $>0$ minus a finite set of points. There is no contradiction with any existing results on absolute sets, including the ones of this article.}
	
 \end{enumerate}
    \end{conj}

The conjecture  was stated in \cite{BW},  only the last part is altered slightly. A3 had been essentially conjectured by Simpson.

It was shown in \cite{BW} that Conjecture A holds if the rank $r$ is $1$, or if $X$ is a complex affine torus or an abelian variety. Moreover, A1 holds if $X$ is projective.

If $M$ is an absolute $\bbQ$-constructible subset of $M_B(X,r)$, then one can pose Conjecture A for $M$, that is with $M$ replacing $M_B(X,r)$ in the statement. This is also the reason why we slightly altered A4 compared with the original statement. The resulting statement is a subconjecture of Conjecture A. Our first result is for $M=M_B^s(X,r)$, the submoduli of simple local systems:

\begin{thm}\label{propA1si}
Consider Conjecture A for $M_B^s(X,r)$ when $X$ is a smooth complex quasi-projective variety. Then A1 holds.
\end{thm}

This implies for example that the irreducible components of $M_B^s(X,r)$ are in bijection with those of $M_B^s(X^\sigma,r)$. At the moment we cannot prove this, nor A1, for $M_B(X,r)$.

Next we address the submoduli $M_B^{cohrig}(X,r)$ of $M_B^s(X,r)$ consisting of {\it cohomologically rigid} simple local systems, when $X$ has a fixed good compactification. These are local systems that  are uniquely, even at infinitesimal level, determined by the conjugacy classes of the local monodromy at infinity and  by the determinant, see Definition \ref{defChrls}.  It follows from \cite{BW} that $M_B^{cohrig}(X,r)$ is absolute $\bbQ$-constructible in $M_B(X,r)$. The main results are:

\begin{thm}\label{thmChR} Consider Conjecture A for $M_B^{cohrig}(X,r)$. Then Conjectures A1, A2, A4 hold. Moreover, Conjecture A3 is equivalent to Simpson's conjecture that  cohomologically rigid simple local systems  with quasi-unipotent determinant and conjugacy classes at infinity are of geometric origin.
\end{thm}

Since Simpson's conjecture is proved in the curve case by  \cite{katz96} and in the rank 2 case by \cite{CS}, a direct consequence is:

\begin{thm}\label{thmCACR}
Conjecture A holds for $M_B^{cohrig}(X,r)$ if $X$ is a curve or if the rank $r=2$. 
\end{thm}

Same proof yields:

\begin{thm}\label{thmP3pts}
Conjecture A holds for $M_B(\bP^1\setminus\{0,1,\infty\},2)$.
\end{thm}

Moreover, we illustrate the intersection of the two cases considered in
Theorem \ref{thmCACR} by giving an explicit description of the moduli of rank $2$ rigid simple local systems on $\bP^1$ minus a finite set of points in Proposition \ref{rk2}.

For the proof of Theorem \ref{propA1si} we use the Riemann-Hilbert correspondence in moduli form. By using the de Rham moduli space of semistable logarithmic connections constructed by  Nitsure \cite{N} and Simpson \cite{S1, S2}, the correspondence takes the form of an analytic map $RH$ from the de Rham moduli to the Betti moduli. We follow \cite{BW} and translate absoluteness in terms of the analytic $RH$ map. However, \cite{BW} overlooked the fact that, beyond rank 1 or the case when $X$ is projective, $RH$ might not be surjective, although we do not know of any such example. For the lack of a proof of surjectivity, we therefore take the time to bring a small list of necessary corrections to \cite{BW}. To finish the proof of Theorem \ref{propA1si}, we then use the fact that $RH$ is a surjective local analytic isomorphism  from a good submoduli of the de Rham moduli onto the submoduli $M_B^s(X,r)$ of simple local systems, an observation going back to Nitsure-Sabbah \cite{NS} and Nitsure \cite{N2}. Next, the defining property of being cohomologically rigid will imply that the eigenvalues at infinity and the determinant give a description up to a quasi-finite morphism of absolute subsets of $M_B^{cohrig}(X,r)$ in terms of torsion-translated complex affine subtori in some $(\bC^*)^N$. We can then finish the proof of Theorem \ref{thmChR} using the results of \cite{BW}. 

A weak and strong analog of Conjecture A2 in an arithmetic context was proposed in \cite{EK2}. Thus \cite[Theorem A]{EK2} which proves the weak arithmetic version for curves bears some similarity to Theorem \ref{thmCACR}. Moreover, \cite[Theorem B]{EK2} which proves the strong arithmetic version, as well as its proof, bears similarity to Theorem \ref{thmP3pts} and its proof. It is however not clear to us if one can pass from the arithmetic or complex version to the other.

\smallskip
\noindent{\bf Acknowledgement.} We thank J. Ant\'onio, H. Esnault, J.-B. Teyssier,  B. Wang, and the referee for useful comments. The  authors were  supported by the grants STRT/13/005 and Methusalem METH/15/026 from KU Leuven, and G097819N and G0F4216N from FWO. H.W. was also sponsored by a CSC PhD scholarship.

\section{Notation}\label{secnot}
 Let $ K\subset \mathbb C$ be a subfield.
    A {\it $\bC$-scheme of finite type defined over $K$} is the base change to $\bC$ of a $K$-scheme of finite type.    A {\it morphism defined over $K$ } between $\bC$-schemes of finite type defined over $K$ is the base change to $\bC$ of a morphism of finite type $K$-schemes. A {\it $\bC$-subscheme is defined over $K$} if the embedding morphism is defined over $K$. Subschemes are assumed to be locally closed in this article.  For a scheme $Y$ over $K$, the set of $\bC$-points $Y(\mathbb C)=\text{Hom}_{Sch(K)}(\text{Spec}(\mathbb C), Y)$ agrees with the set of $\bC$-points of the base change of $Y$ to $\bC$. If $Y$ is of finite type over $K$, a {\it $K$-constructible} (respectively, {\it $K$-closed}) subset of $Y(\bC)$ is a finite union of sets of  $\mathbb C$-points of $K$-subschemes  (respectively, closed $K$-subschemes) of $Y$. A {\it $K$-pseudo-morphism} between between two $K$-constructible sets is a map whose graph is $K$-constructible. Thus two $K$-constructible sets are {\it $K$-pseudo-isomorphic} if there exist finite bijective partitions of each into sets of $\bC$-points of $K$-subschemes such that two parts corresponding to each other are isomorphic $\bC$-schemes and the isomorphism is defined over $K$.   A $\bZ$-valued function on a $K$-constructible set is a {\it $K$-constructible function} if it has finite image and the inverse image of every $k\in\bZ$ is $K$-constructible.

 Let $F$ be a functor from the category of $K$-schemes to the category of sets. A $K$-scheme $Y$ is a {\it fine moduli space for $F$} if it {\it represents} $F$, that is, there is a natural isomorphism of functors $F\ra \Hom(\_,Y)=:Y(\_)$. A $K$-scheme $Y$ {\it corepresents} $F$ if there exists a natural transformation $F\ra Y(\_)$ such that for every natural transformation $F\ra Y'(\_)$ into a $K$-scheme $Y'$, there is a factorization through a unique morphism of $K$-schemes $Y\ra Y'$. A $K$-scheme $Y$ is a {\it coarse moduli space for $F$} if it corepresents $F$ and the natural transformation gives a bijection $F(L)=Y(L)$ for every algebraically closed field extension $L$ of $K$. A scheme $Y$ {\it universally corepresents} $F$ if for every morphism of schemes $Z\ra Y$, $Z$ corepresents the fiber product functor $Z(\_)\times_{Y(\_)} F$.
 
If $G$ is a reductive algebraic group acting on a $\bC$-scheme $Y$, and if $Y\ra X$ is a $G$-invariant morphism of $\bC$-schemes, then $X$ is a {\it (universal) categorical quotient} if $X$ together with the induced natural transformation $F\ra X$ (universally) corepresents the quotient functor $F(T)=Y(T)/G(T)$ on $\bC$-schemes.

    For a $\bC$-scheme of finite type $X$, we denote by $X^{an}$ the associated complex analytic space and by $X(\bC)^{an}$  the associated reduced complex analytic space. An analytic subspace of $X$ will always mean one of $X^{an}$. If $X$ is reduced, $X^{an}=X(\bC)^{an}$. 
    
    Varieties will always be assumed irreducible.

\section{Betti moduli spaces}\label{defbetti} 
In this section we recall  basic facts about Betti moduli spaces. Let $X$ be a smooth complex algebraic variety.  We fix a base point $x_0\in X$. By the fundamental group $\pi_1(X,x_0)$ and by local systems (always finite-dimensional complex unless mentioned otherwise) on $X$, we always mean with respect to the underlying Euclidean topology, that is, on $X^{an}$.

Let $r>0$ be an integer. The {\it space of representations of rank $r$} of the  fundamental group $\pi_1(X,x_0)$ is 
	\[ R_B(X,x_0,r)\colonequals  \text{Hom}(\pi_1(X,x_0), \text{GL}_r),\]
the $\bC$-scheme representing the functor associating to a $\bC$-scheme $T$  the set  $$\Hom(\pi_1(X,x_0),\GL_r(\Gamma(T,\cO_T))).$$

	Since complex algebraic varieties have the homotopy type of finite-type CW complexes, the group $\pi_1(X,x_0)$ is finitely presented.
	Fixing generators $\gamma_1,\dots, \gamma_m$ of  $\pi_1(X,x_0)$, one has an embedding of the representation space as a closed affine subscheme of $\GL_r^{\times m}$ defined over $\bQ$, which on the complex points is		
	\begin{equation}\label{(2.1)}R_B(X,x_0,r)(\mathbb C)
	\hookrightarrow \text{GL}_r(\mathbb C)^{\times m},\;\;
	\rho\mapsto (\rho(\gamma_1),\dots, \rho(\gamma_m))\end{equation}
and the equations satisfied come from the relations among the generators.

The representation space $R_B(X,x_0,r)$ can also be interpreted as the fine moduli space of framed local systems of rank $r$ on $X$. A family over a $\bC$-scheme $T$ of framed local systems on $X$  is a locally  constant sheaf $L$ of $\Gamma(T,\cO_T)$-modules on $X^{an}$, together with an isomorphism $\xi:L|_{x_0}\xa{\sim}\Gamma(T,\cO_T)^{\oplus r}$.

	The group  $\text{GL}_r(\bC)$ acts on $R_B(X,x_0,r)$ by simultaneous conjugation.
	The \emph{moduli space of local systems of rank $r$}  is defined as the affine quotient of $R_B(X,x_0,r)$ with respect to this action,
	\[ M_{B}(X,r)=R_B(X,x_0,r)\sslash \text{GL}_r\coloneqq \text{Spec}\left(H^0\left(R_B(X,x_0,r), \mathcal O_{R_B(X,x_0,r)}\right)^{\text{GL}_r(\bC)}\right).\]
This is a universal categorical quotient, see  \cite[Proposition 6.1]{S2}.	The embedding (\ref{(2.1)}) descends to a closed embedding of affine schemes defined over $\bQ$	
	\begin{equation}\label{eqCl} 
	M_B(X,r) \hookrightarrow \text{GL}_r^{\times m}\sslash {\GL} _r.
	\end{equation}
The complex points of $M_B(X,r)$ are in one-to-one correspondence with the isomorphism classes of complex semisimple local systems of rank $r$ on $X$. We recall that a local system $L$ is called simple, or irreducible, if the associated representation $\pi_1(X,x_0)\ra\GL(L|_{x_0})$ is so, and $L$ is called semisimple if it is a direct sum of simple local systems. We denote by 
$$
q_B:R_B(X,x_0,r)\ra M_{B}(X,r)
$$
the quotient morphism. The fiber over a complex point $L\in M_B(X,r)(\bC)$ consists of all framed  $\bC$-local systems with semi-simplification isomorphic to $L$.

For a representation  $\rho\colon \pi_1(X,x_0)\to \text{GL}_r(\mathbb C)$ we denote the associated local system on $X$ by $ L_\rho$.  The corresponding class in $M_B(X,r)(\bC)$ will be sometimes denoted $[\rho]$.

The quotient map restricts to a morphism
$
q_B:R^s_B(X,x_0,r)\to M^s_B(X,r) 
$
on the open subschemes of $R_B(X,x_0,r)$ and $M_B(X,r)$, respectively, whose closed points correspond to simple local systems, since $R^s_B(X,x_0,r)$ is also the GIT-stable locus. 
	
We define  $$R_B(X,x_0):=\coprod_{r}R_B(X,x_0,r)\quad\text{and}\quad M_B(X):=\coprod_rM_B(X,r),$$ the disjoint unions of schemes.

\begin{example}\label{exClCurve}\label{C^*} $\;$

	(1) If $X=\bP^1\setminus\{D_1,\ldots, D_{m+1}\}$, the complement of $m+1>1$ distinct points, then (\ref{eqCl}) is an isomorphism.
	 
 (2)	(a) (\cite[Example 1.35]{ lubotzky1985varieties} ) If $x_{11}, x_{12}, x_{21}, x_{22}, x_{31}, x_{32}$ denote the pairwise-unordered coordinates on $((\bC^*)^{(2)})^3$, then one has isomorphisms
 $$
	\text{GL}_2(\mathbb C)^{\times 2}\sslash\text{GL}_2(\mathbb C)\xa{\sim}  V(x_{11}x_{12}x_{21}x_{22}x_{31}x_{32}-1) \xa{\sim}  \mathbb C^3\times (\mathbb C^*)^2
	$$
 with composition 
 \[[(g_1, g_2)]\mapsto (\tr(g_1), \tr(g_2), \tr(g_1g_2), \det(g_1)^{-1}, \det(g_2)^{-1}).\]

(b) (\cite{HW}) With $X=\bP^1\setminus\{0,1,\infty\}$, then
$$M_B(X,2)\setminus M_B^s(X,2)\simeq V(x_{11}x_{12}x_{21}x_{22}x_{31}x_{32}-1)\cap V\left(\prod_{i,j,k=1,2}(x_{1i}x_{2j}x_{3k}-1)\right).$$

\end{example}

\section{Absolute sets}

Absoluteness is an abstract categorical notion introduced in \cite{BW}. The absolute sets of local systems are subject to Conjecture A from the Introduction, also called the  Special Varieties Package Conjecture in \cite{BW}. The Betti moduli spaces serve as concrete ambient spaces for absolute sets of local systems.  We recall in this section the definition of absolute sets and some of their properties.

\subsection{Unispaces.}
One has first a categorical notion of constructible (respectively, closed) subset of $S$ of $\sC(\bC)$ relative to a functor  $$\sC: Alg_{ft, reg}(\bC)\to Set$$
from the category of finite type regular $\bC$-algebras to sets. Such $\sC$ is called a {\it unispace}. If $R\to R'$ is a morphism in $Alg_{ft, reg}(\bC)$ and $\cF _R\in \sC(R)$, the image of $\cF_R$ in $\sC(R')$ under the morphism $\sC(R)\to\sC(R')$ will be denoted $\cF_R\star_RR'$. We can think of $\cF_R$ as a family of objects from $\sC(\bC)$ parametrized by $\Spec(R)$.

Then $S\subset \sC(\bC)$ is called {\it constructible (respectively, closed) with respect to the unispace $\sC$} if for every  $R\in Alg_{ft, reg}(\bC)$ and every $\cF_R\in\sC(R)$, the subset of maximal ideals
$$
\{m \in \Spec(R)(\bC)\mid \cF_R\star_RR/m\in S\}
$$
is Zariski constructible (respectively, closed) in the complex variety associated to $R$. 

If $\sC'\to\sC$ is a natural transformation between unispaces and $S\subset\sC(\bC)$ is closed with respect to $\sC$, then the inverse image of $S$ in $\sC'(\bC)$ is closed with respect to $\sC'$.

One can also define the notion of a {\it morphism of unispaces} $\sC'\to\sC$ such that the inverse image of a constructible $S\subset \sC(\bC)$ is constructible with respect to $\sC'$,  and such that when $\sC, \sC'$ are represented by finite type $\bC$-schemes, a morphism of unispaces is equivalent to a pseudo-morphism, \cite[Section 2]{BW}. Moreover, there is a natural bijection between functions $\sC(\bC)\to \bZ$ that are constructible with respect to the unispace $\sC$ and morphisms from $\sC$ to the  unispace of $\bZ$-valued constructible functions
$$
\ul{\bZ}^{cstr}: R\mapsto\{\text{constructible maps }\Spec(R)(\bC)\to \bZ\}$$
with $\star_RR'$ being the composition with $\Spec(R')(\bC)\to\Spec(R)(\bC)$.

The simplest non-trivial unispace is the unispace of finite free modules
$$
Mod_{free}: R\mapsto \{\text{finite free }R-\text{modules}\}/_\simeq
$$
where $\star_RR'$ is the usual tensor product $\otimes_RR'$. In fact composition of the rank map on $Mod_{free}(\bC)$ with any bijection $\bN\xa{\sim}\bZ$ is a constructible function with respect to $Mod_{free}$, and it corresponds to an isomorphism of unispaces $Mod_{free}\xa{\sim}\ul{\bZ}^{cstr}$.

Thus for any map $f:\sC(\bC)\to\bZ$ which comes from a morphism of unispaces $\sC\to Mod_{free}$, and for any $k\in\bZ$, the subset $f^{-1}(k)$ is constructible with respect to $\sC$.

For a complex algebraic variety $X$ there is a natural transformation of unispaces
\be\label{eqIi}
LocSys_{free}(X,\_) \to D^b_c(X,\_).
\ee
The first unispace associates to $R\in Alg_{ft,reg}(\bC)$ the isomorphism classes of local systems $\cL_R$ of finite free $R$-modules on $X$, and for a morphism $R\to R'$ the operation $\star_RR'$ is the usual tensor product $\otimes_RR'$. The second unispace associates to $R$ the isomorphism classes of objects in the bounded derived category of constructible sheaves of finite type $R$-modules on $X^{an}$, and the operation $\star_RR'$ is the derived tensor product $\mathop{\otimes}^L_RR'$. The stability of $D^b_c(X,\_)$ under the derived tensor product is precisely the reason why in the definition of unispaces one restricts to regular and not arbitrary finite type $\bC$-algebras. The above natural transformation sends $\cL_R$ to itself as a complex concentrated only in degree 0.

It turns out that ``all" natural functors on derived categories $D^b_c(X,\bC)$ extend to give morphisms of unispaces, \cite[Section 5]{BW}. They are  {\it constructible functors}, that is, via taking inverse image they preserve constructibility with respect to the fixed unispaces. In fact most of the times these functors give natural transformations of unispaces, so they are {\it continuous functors}, that is, via taking inverse image they preserve closedness with respect to the fixed unispaces. ``All" here means a list of commonly used functors, see \cite[Section 5]{BW} and \cite{B-all}.

 One produces in this way many  subsets of $LocSys_{free}(X,\bC)=LocSys(X,\bC)$ 
that are constructible or closed with respect to the unispace $LocSys_{free}(X,\_)$ since
$
LocSys_{free}(pt,\_) = Mod_{free}. 
$
For any complex variety $Y$ with a natural transformation
$
Y(\_)\to LocSys_{free}(X,\_),
$
one produces thus constructible or Zariski closed subsets of $Y(\bC)$ via inverse image. Taking $Y=R_B(X,x_0)$ and projecting further via the quotient map $q_B:R_B(X,x_0)\to M_B(X)$, one produces constructible or Zariski closed subsets of $M_B(X)(\bC)$. For example, the cohomology jump locus
$$
\{L\in M_B(X,r)(\bC)\mid \dim H^i(X,L)\ge k\}
$$ 
is Zariski closed in $M_B(X,r)$ for all $r, i, k\in\bN$, since it is obtained from the subset
$$
\{L\in LocSys(X,\bC)\mid \dim H^i(X,L)\ge k\}
$$ 
which is closed with respect to the unispace $LocSys_{free}(X,\_)$. 
 
If $K$ is a subfield of $\bC$, all the notions from  above, which correspond to the case $K=\bC$, can be updated to be defined over $K$. The category of $\bC$-unispaces defined over $K$ contains as a full subcategory the category of $\bC$-schemes of finite type defined over $K$ together with $K$-pseudo-morphisms. For example, the unispaces from (\ref{eqIi}) are defined over $\bQ$. One produces this way many subsets of $LocSys(X,\bC)$ that are $\bQ$-constructible or $\bQ$-closed with respect to the unispace $LocSys_{free}(X,\_)$, if one uses natural functors that are defined over $\bQ$ in the above sense, otherwise the field of definition has to be enlarged. For example, the operation on $D^b_c(X,\bC)$ given by the derived tensor with the complexification $M\otimes_K\bC$ of a $K$-local system $M$ lifts to a morphism defined over $K$, but not over a strictly smaller subfield of $K$, of unispaces defined over $K$. In fact, $(\_)\otimes_\bC^L M$ is a $K$-continuous functor on $D^b_c(X,\bC)$ with respect to the unispace $D^b_c(X,\_)$ defined over $K$, \cite[Theorem 5.14.3]{BW}. This gives for example that
$$
\{L\in M_B(X,r)(\bC)\mid \dim H^i(X,L\otimes_KM)\ge k\}
$$
 is Zariski $K$-closed but not $\bQ$-closed if $\bQ\neq K$.

\subsection{Absolute sets.} We assume now that $X$ is a smooth complex algebraic variety.
For each $\sigma\in \text{Gal}(\mathbb C/\mathbb Q)$, let $X^\sigma\ra X$ be the base change of $X$ via $\sigma$. This map  is typically not continuous in the Euclidean topology since the only continuous automorphisms of $\bC$ are the identity and the complex conjugation. However it is a homeomorphism for the Zariski topology, and it is an isomorphism of $\bQ$-schemes. Consider the diagram
 \be\label{eqDrHg}
\begin{gathered}
\xymatrix{
   	D^b_{rh}(\sD_X)   \ar[d]_{RH}  \ar[rr]^{p_\sigma} & & D^b_{rh}(\sD_{X^\sigma})\ar[d]_{RH}\\
   	 D^b_c(X,\bC)&  &  D^b_c(X^\sigma,\bC).
}
\end{gathered}
\ee
Here $D^b_{rh}(\sD_X)$ is the bounded derived category of algebraic regular holonomic $\sD_X$-modules, $RH$ is the  Riemann-Hilbert equivalence, and $p_\sigma$  is the pullback of $\sD_X$-modules under the base change over $\sigma$, also an equivalence of categories. On isomorphism classes the maps induced by the diagram (\ref{eqDrHg}) are bijections.

For a subset $T$ of the isomorphism classes of objects in $D^b_c(X,\bC)$, define
\be\label{eqSss}
T^\sigma := RH\circ p_\sigma\circ RH^{-1} (T), 
\ee
a subset of the  isomorphism classes of objects in $D^b_c(X^\sigma,\bC)$.

\begin{defn}\label{defAk} Let $\sC:\sigma\mapsto \sC_{X^\sigma}$ be the assignment to every $\sigma\in Gal(\bC/\bQ)$ of a sub-unispace of $D^b_c(X^\sigma,\_)$ defined over $K$, see \cite[3.3]{BW}. We also assume that for all $\sigma\in Gal(\bC/\bQ)$,  $\sC_X(\bC)^\sigma =\sC_{X^\sigma}(\bC)$. Then a subset $T\subset \sC_X(\bC)$ is {\it absolute $K$-constructible (respectively, absolute $K$-closed) with respect to  $\sC$}, or simply {\it with respect to the unispace $\sC_X$} if the assignment $\sC$ is clear,  if the set 
$T^\sigma \subset \sC_{X^\sigma}(\bC)$ is $K$-constructible (respectively, $K$-closed) with respect to $\sC_{X^\sigma}$ for all $\sigma\in Gal(\bC/\bQ)$. 




\end{defn}

In practice, ``all" natural functors on  (bounded derived) categories of constructible sheaves on $X^{an}$ have  algebraic $\sD$-module counterparts. Their source and target dictate in each case the choice of unispace $\sC_X$.  Compatibility with the algebraic counterparts ensures that absolute $K$-constructibility (respectively, absolute $K$-closedness) is preserved by taking inverse image, in which case the functors are called {\it absolute $K$-constructible (respectively, absolute $K$-continuous) functors}. Here $K$ is mostly $\bQ$, unless we deal with a functor like $(\_)\otimes_\bC^LM$ above with $M$ defined over $K$ instead of $\bQ$, see \cite[Section 6]{BW}.

\subsection{Absolute sets of local systems}\label{subAsl} Let $X$ be a smooth complex algebraic variety.   Denote by $$
l: R_B(X,x_0)(\bC)\to LocSys(X,\bC)
$$
the map forgetting the frame and keeping the isomorphism class of the local system. Note that $l$ lifts to a natural transformation 
$
l: R_B(X,x_0)\to LocSys_{free}(X,\_)
$
defined over $\bQ$ of unispaces defined over $\bQ$. 

Recall that there is the GIT quotient map
$
q_B: R_B(X,x_0)\to M_B(X).
$
Then $q_B(l^{-1}(T))\subset M_B(X)(\bC)$ is the set of isomorphism classes of the semisimplifications of the local systems in $T$ if $T\subset LocSys(X,\bC)$.


If $M\subset M_B(X)(\bC)$, we let $$\tilde M:=l(q_B^{-1}(M))\subset LocSys(X,\bC)$$ be the subset consisting of isomorphism classes of local systems with semisimplifications in $M$.  We define the unispace
$$
\sM_X : R\mapsto \{\cL_R\in LocSys_{free}(X,R)\mid \cL_R\otimes_R R/m\in \tilde M,\; \forall m\in\Spec(R)(\bC)\}
$$
together with the usual base change. Then $\sM_X(\bC)=\tilde M$ and $\sM_X$ is a sub-unispace of $LocSys_{free}(X,\_)$ defined over $K$. If $M=M_B(X)(\bC)$, then $\sM_X=LocSys_{free}(X,\_)$. We let $\sM_{X^\sigma}$ be unispace associated to $M^\sigma$, where $M^\sigma$ is defined as in (\ref{eqSss}). Note that $(\_)\,\tilde{  }$ and $(\_)^\sigma$ commute. The assignment $\sM:\sigma\mapsto \sM_{X^\sigma}$ allows us to define what it means for  $S\subset \tilde M$ to be an  absolute $K$-constructible (respectively, absolute $K$-closed) with respect to $\sM$, by Definition \ref{defAk}.

\begin{defn}\label{defBWas}  Let $M\subset M_B(X)(\bC)$. A subset $S$ of  $M$ is called {\it absolute $K$-constructible (respectively, absolute $K$-closed) in $M$} if $S=q_B(l^{-1}(T))$ for some subset $T$ of $\tilde M$ such that $T$ is absolute $K$-constructible (respectively, absolute $K$-closed) with respect to the unispace $\sM_X$. If $M=M_B(X)(\bC)$, we simply say that $S$ is {\it absolute $K$-constructible (respectively, absolute $K$-closed)}, cf.  \cite[Definition 7.3.1]{BW}. 
\end{defn}

We will show the following unispace-free characterization:

\begin{prop}\label{lemAnlg} Let $M\subset M_B(X)(\bC)$ be  absolute $K$-constructible.
Let $S\subset M$. Then $S$ is absolute $K$-constructible (respectively, absolute $K$-closed) in $M$ iff $S^\sigma$ is $K$-constructible (respectively, $K$-closed) in $M^\sigma$ for all $\sigma\in Gal(\bC/\bQ)$. 
\end{prop}

\begin{rmk} Note that taking $M=M_B(X)(\bC)$ we have a characterization of absolute subsets in $M_B(X)(\bC)$. However, without the unispace approach there was a priori no reason why one can produce subsets with this characterization from general functors on derived categories which might lack Betti or de Rham moduli spaces. 
\end{rmk}

Since (\ref{eqDrHg}) induces bijections on isomorphism classes, a consequence of the proposition is:

\begin{lemma}\label{lemOps} The following hold in $M_B(X,r)(\bC)$:

- The intersection of (possibly infinitely-many) absolute $K$-closed subsets  is absolute $K$-closed.

- The union of finitely-many absolute $K$-closed subsets is absolute $K$-closed. 

- The intersection and the union of finitely-many absolute $K$-constructible subsets is absolute $K$-constructible. 

- The complement of an absolute $K$-constructible subset in an absolute $K$-constructible subset is an absolute $K$-constructible set.
\end{lemma}

Hence one can define a topology depending on $K$ on $M_B(X,r)(\bC)$ whose closed sets are the absolute $K$-closed subsets. 
It is part of the conjecture on absolute sets, see \ref{secSvPc}, that the family of absolute $\bar{\bQ}$-constructible subsets is equal to the subfamily of constructible sets in the sense of this new topology for $K=\bar\bQ$.

As a consequence of Proposition \ref{lemAnlg}, one obtains the following relation between absoluteness in $M$ with absoluteness in $M_B(X)(\bC)$:

\begin{lemma}\label{lemLll2}$\;$

(1) If $S\subset M_B(X)(\bC)$ is absolute $K$-constructible (respectively absolute $K$-closed) in $M_B(X)(\bC)$, then the intersection $S\cap M$ is  absolute $K$-constructible (respectively $K$-closed) in $M$.

(2) Let $M\subset M_B(X)(\bC)$ be  absolute $K$-constructible. If $S\subset M$, then it is an absolute $K$-constructible subset of $M$ iff it is absolute $K$-constructible subset of $M_B(X)(\bC)$.

\end{lemma}

We prove Proposition \ref{lemAnlg} after a few lemmas.

\begin{lemma}\label{lemCrT} Let $T\subset LocSys(X,\bC)$, $S=q_B(l^{-1}(T))$, and $\tilde T=l(q_B^{-1}(S))=\tilde S$.

(1) If $T$ is $K$-closed with respect to $LocSys_{free}(X,\_)$ then $T$ contains all the semisimplifications of its elements, and $S$ is $K$-closed in $M_B(X)(\bC)$. 

(2) If $S$ is $K$-constructible (respectively, $K$-closed) in $M_B(X)(\bC)$, then $\tilde T$ is $K$-constructible (respectively, $K$-closed) with respect to $LocSys_{free}(X,\_)$.

(3) If $T$ is $K$-constructible with respect to $LocSys_{free}(X,\_)$ then $S$ is $K$-constructible in $M_B(X)(\bC)$. 

(4) If $T$ is $K$-constructible (respectively, $K$-closed) with respect to $LocSys_{free}(X,\_)$ then $\tilde T$ is also.

\end{lemma}
\begin{proof}
(1) The natural transformation $l$ preserves $K$-closedness after taking inverse image. Thus $l^{-1}(T)$ is closed. It is moreover  invariant under the group action on $R_B(X,x_0)(\bC)$ since it contains all possible frames for each local systems in $T$. Thus the orbit of every point in $l^{-1}(T)$ is contained in $l^{-1}(T)$, and is closed. The claim follows.

(2)   Let $\cL_R$ be a local system on $X$ of rank $r$ free modules over a finite type regular $\bC$-scheme $R$ defined over $K$. We need to show that $$\tilde T_R=\{m\in \Spec (R)(\bC)\mid \cL_R\otimes_RR/m\in \tilde T\}$$
 is $K$-constructible (respectively, $K$-closed) in $\Spec(R)(\bC)$. Choose a frame $\cL_R|_{x_0}\xa{\sim}R^{\oplus r}$. Since $R_B(X,x_0,r)$ is a fine moduli space for framed local systems, there is a morphism 
\be\label{eqf_R}
f_R:\Spec(R)(\bC)\to R_B(X,x_0,r)(\bC)
\ee
 defined over $K$, mapping a maximal ideal $m$ to $\cL_R\otimes_RR/m$ together with the induced frame. By construction, $\tilde T_R=f_R^{-1}(q_B^{-1}(S))$ and hence it is $K$-constructible (respectively, $K$-closed).

(3) The natural transformation $l$ preserves $K$-constructibility after taking inverse image, and thus $S$ is $K$-constructible too.

(4) It follows from (3) (respectively, (1)) together with (2).
\end{proof}

One obtains directly:
 
 \begin{lemma}\label{lemMAA}
Let $S\subset M_B(X,r)(\bC)$ and $\tilde S=l(q_B^{-1}(S))\subset LocSys(X,\bC)$. Then  $\tilde S$ is $K$-constructible (respectively, $K$-closed) with respect to the unispace $LocSys_{free}(X,\_)$ iff $S$ is $K$-constructible (respectively, $K$-closed) in $M_B(X,r)(\bC)$. 
\end{lemma}

We apply the previous two lemmas to prove a generalization of the last lemma:

\begin{lemma}\label{lemMAA-2}
Let $N\subset LocSys(X,\bC)$ be $K$-constructible with respect to  $LocSys_{free}(X,\_)$. Let $M=q_B(l^{-1}(N))$. Let $S\subset M$ and $\tilde S=l(q_B^{-1}(S))$. Then  $\tilde S$ is $K$-constructible (respectively, $K$-closed) with respect to the unispace $\sM_X$ iff $S$ is $K$-constructible (respectively, $K$-closed) in $M$. 
\end{lemma}
\begin{proof}
By Lemma \ref{lemCrT}, $M$ is $K$-constructible in $M_B(X)(\bC)$ and $\tilde M$ is $K$-constructible with respect to  $LocSys_{free}(X,\_)$.

Suppose that $\tilde S$ is $K$-constructible (respectively, $K$-closed) with respect to $\sM_X$. 
Since $R_B(X,x_0)$ is a fine moduli space, $l$ restricts to a natural transformation $q_B^{-1}(M)(\_)\to \sM_X$. This implies that $l^{-1}(\tilde S)$ is $K$-constructible (respectively, $K$-closed) in $q_B^{-1}(M)$. Thus the image $S=q_B(l^{-1}(\tilde S))$ is $K$-constructible in $M$ (respectively, $K$-closed, since $l^{-1}(\tilde S)=q_B^{-1}(S)$ and $q_B$ is a universal categorical quotient). 

For the converse, note that by Lemma \ref{lemMAA}, $S$ is $K$-constructible  in $M$ iff $\tilde S$ is $K$-constructible with respect to the unispace $LocSys_{free}(X,\_)$.  Since $\tilde S\subset \tilde M$, we have by definition of $\sM_X$ that if $\tilde S$ is $K$-constructible (respectively, $K$-closed) with respect to $LocSys_{free}(X,\_)$ then it is also so with respect to $\sM_X$.

Lastly, if $S$ is $K$-closed  in $M$, write $S=M\cap S'$ where $S'$ is $K$-closed in $M_B(X)(\bC)$. Then by Lemma \ref{lemMAA}, $\tilde S'=l(q_B^{-1}(S'))$ is $K$-closed with respect to  $LocSys_{free}(X,\_)$. Since $\tilde M$ is $K$-constructible with respect to  $LocSys_{free}(X,\_)$, we have that $\tilde S =\tilde M \cap \tilde S'$ is $K$-closed with respect to $\sM_X$, by the definition of $\sM_X$.
\end{proof}

\begin{lemma}\label{lemTil} Let $M\subset M_B(X)(\bC)$ be  absolute $K$-constructible.
A subset $S$ of $M$ is  absolute $K$-constructible (respectively, absolute $K$-closed) in $M$ iff the subset $\tilde{S}=l(q_B^{-1}(S))$  of $\tilde M$ is   absolute $K$-constructible (respectively, absolute $K$-closed)  with respect to the unispace $\sM_X$.
\end{lemma}
\begin{proof}
The implication $\Leftarrow$ follows by definition. We prove the implication $\Rightarrow$. 

Let $T$ be a subset of $\tilde M$ that is absolute $K$-constructible (respectively, absolute $K$-closed) with respect to the unispace $\sM_X$, and such that $S=q_B(l^{-1}(T))$. Then $T^\sigma$ is $K$-constructible (respectively,  $K$-closed) with respect to the unispace $\sM_X$ for all $\sigma$, by definition. Note that  $S^\sigma=q_B(l^{-1}(T^\sigma))$, and, as remarked before, $\widetilde{S^\sigma}:=l(q_B^{-1}(S^\sigma))$ is precisely $(\tilde S)^\sigma$. It follows that $(\tilde S)^\sigma$ is $K$-constructible  with respect to $LocSys_{free}(X^\sigma,\_)$ by Lemma \ref{lemCrT} (4), and hence also with respect to $\sM_{X^\sigma}$ by the definition of the later. This proves the constructible case.

 It remains to prove the closed case. Assume that $T$ is absolute $K$-closed with respect to $\sM_X$. Since $T$ is $K$-closed with respect to $\sM_X$, $l^{-1}(T)$ is $K$-closed in $q_B^{-1}(M)$ and invariant under the group action since it contains all possible frames for each local systems in $T$.  Thus $S=q_B(l^{-1}(T))$ is $K$-closed in $M$. By Lemma \ref{lemMAA-2} it follows that $\tilde S$ is $K$-closed with respect to $\sM_X$. Applying the same argument to $T^\sigma$, we have that $\tilde S^\sigma$ is $K$-closed with respect to $\sM_{X^\sigma}$. This shows that $\tilde S$ is absolute $K$-closed in $M$ with respect to $\sM_X$.
\end{proof}

\begin{proof}[Proof of Proposition \ref{lemAnlg}] If follows from  Lemma \ref{lemTil} and Lemma \ref{lemMAA-2}, by noting that with  $S$, $\tilde S$ as in Lemma \ref{lemTil}, one has $S^\sigma = q_B(l^{-1}(\tilde S^\sigma))$ and $\tilde S^\sigma=l(q_B^{-1}(S^\sigma))$.
\end{proof}

\subsection{Special Varieties Package Conjecture}\label{secSvPc} With all notions in hand now one can make fully sense of  the statement of Conjecture A, consisting of parts A1-A4, from Introduction. We recall some partial results and pinpoint some difficulties. On the positive side one has:

\begin{theorem}\label{budurwang} (\cite[Theorem 1.3.1, Proposition 10.4.3]{BW})

Conjecture A holds if $r=1$. Moreover, every absolute $\bar{\mathbb Q}$-closed subset of the affine algebraic torus $ M_B(X,1)(\mathbb C)$  is a finite union of torsion-translated affine  algebraic subtori.    
\end{theorem}

In fact, the proof gives the stronger statement that the subtori involved are motivic, that is, image of the pullback map $M_B(G,1)\ra M_B(X,1)$ associated to some morphism $X\ra G$ to a semi-abelian variety $G$. This was also proved later differently by \cite{EK} for an arithmetic counterpart of absolute sets; see also \cite{L} for an extension to a singular setup.

\begin{theorem} (\cite[Propositions 10.4.4, 10..4.5]{BW}) 

(1) Conjecture A holds in any rank if $X$ is a complex affine torus or an abelian variety. 

(2) A1 holds if $X$ is projective.
\end{theorem}



An essential difficulty with A1 is that the following concrete question is open in general, although we expect a positive answer:

\begin{que}\label{queIrr} Let $X$ be a smooth non-proper complex algebraic variety and $\sigma\in Gal(\bC/\bQ)$.  If $M_B(X,r)$ is irreducible, is $M_B(X^\sigma,r)$ also irreducible? If $M_B(X,r)$ is not irreducible, and $T$ is an irreducible component of $M_B(X,r)$, is $T^\sigma$ at least constructible? Are the irreducible components in bijection with those of $M_B(X^\sigma,r)$? 
\end{que}

Note that one can formulate by analogy the following:

\begin{conj}[Conjecture A for $M$] \label{conj1BM}   Let $M$ be an absolute $\bbQ$-constructible subset of $M_B(X,r)(\bC)$. Then the conditions A1-A4 hold in $M$.    
\end{conj}     
      
This is not just an analogy. If A1 holds in $M_B(X)(\bC)$, it would imply that the absolute $\bbQ$-closed subsets of $M$ are also restrictions to $M$ of absolute $\bbQ$-closed subsets of $M_B(X)(\bC)$. Thus Conjecture A for $M_B(X,r)$ implies it for any absolute $\bbQ$-constructible subset $M$, and more generally one has the following whose proof is straight-forward:

\begin{lemma}
Conjecture A for $M$ implies Conjecture A for $M'$, for any two absolute $\bbQ$-constructible subsets $M'\subset M$ of $M_B(X,r)$.
\end{lemma}           

In this article we will (partly) address Conjecture A for $M$, when $M$ is $M_B^s(X,r)(\bC)$ and $M_B^{cohrig}(X,r)(\bC)$.

\section{Absoluteness via de Rham moduli spaces}  

The absolute sets we are interested in live on  Betti moduli spaces. The notion of absoluteness is however defined via the categorical Riemann-Hilbert correspondence. Using de Rham moduli spaces one has in certain situations a more concrete characterization of absoluteness, according to \cite{BW}, which we recall, correct, and expand in detail in \ref{defderham}-\ref{subAbs}. The concise list of corrections to \cite{BW} is placed as an appendix to this section in \ref{secApp}. In \ref{subslli} we use this characterization, together with the  local isomorphism property of the analytic Riemann-Hilbert map from a good de Rham moduli subspace onto the Betti submoduli of simple local systems, to prove Theorem \ref{propA1si}. In \ref{subMinf} we show that the map taking determinant and eigenvalues at infinity preserves absoluteness in a certain sense. This will be used in Section \ref{secRls} on rigid local systems.

\subsection{De Rham moduli spaces}\label{defderham} We recall the construction and the basic facts on de Rham moduli spaces. Let $X$ be a smooth complex quasi-projective variety. Fix a point $x_0$ in $X$,   an embedding of $X$ into a smooth complex projective variety $\bar{X}$ such that the complement $D=\bar{X}\setminus X$ is a divisor with normal crossing singularities, and a very  ample line bundle on $\bar{X}$. Let $k$ be a common field of definition of $X, \bar{X}, D, x_0$.

		The {\it de Rham representation space}   is a disjoint union $k$-schemes of finite type  $$R_{DR}(\bar{X}/D, x_0, r)=\coprod_P R_{DR}(\bar{X}/D, x_0, P),$$ 
each of which is a fine moduli space for framed semistable logarithmic connections of rank $r$ and Hilbert polynomial $P$. We explicit now what this means. 

All the definitions below should be relative to $k$, although we largely  suppress this from notation for simplicity, e.g. $X$ is now a $k$-variety.

A {\it logarithmic connection}  is a flat connection $\nabla:E\ra E\otimes_{\cO_{\bar{X}}}\Omega^1_{\bar{X}}(\log D)$ with logarithmic poles along $D$ on a $\mathcal O_{\bar{X}}$-coherent torsion-free module $ E$. This places a structure of $\Lambda$-module on $E$, where $\Lambda=\sD_{\bar{X}}(\log D)$ is the sheaf of logarithmic differential operators. The restriction $E|_X$ is a locally free $\cO _X$-module since it is an $\cO _X$-coherent $\sD_X$-module. Hence $(E,\nabla)|_X$ is a vector bundle with a flat connection on $X$.

A logarithmic connection $(E,\nabla)$ is {\it semistable} (respectively, {\it stable}) if for any proper $\cO_{\bar{X}}$-coherent $\Lambda$-submodule $F$ of $E$, $p(F,m)/rk(F)\le p(E,m)/rk(E)$ (respectively, $<$) for  all integers $m\gg 0$, where $p(F,t)$ is the Hilbert polynomial of $F$ with respect to the fixed polarization on $\bar{X}$, and $rk(F)$ is the rank.

A {\it frame} at $x_0$ of $E$ is an isomorphism $\xi \colon  E|_{x_0}\xa{\sim} \mathcal O_{\bar{X}}^{\oplus r}|_{x_0}$. 
		 
There are straight-forward notions of a family $(E_T,\nabla_T,\xi_T)$ over a  $k$-scheme $T$ of framed semistable logarithmic connections on $(\bar{X},D)$ with fixed Hilbert polynomial, and of equivalence and pullbacks of families, see  \cite[p.92]{S1}, \cite[\S 2]{N}. In this case, $(E_T,\nabla_T)$ is in particular a logarithmic connection on $(\bar{X}\times_k T, D\times_k T)$ relative to $T$. Then $R_{DR}(\bar{X}/D, x_0, P)$ represents the functor associating to $T$ the equivalence class of families over $T$ of framed semistable logarithmic connections with Hilbert polynomial $P$, see  \cite[Theorem 4.10]{S1}.

	There is an action of $\text{GL}_r(\mathbb C)$ on $R_{DR}(\bar{X}/D, x_0, P)$  given by the change of basis in the frame and all the points are GIT-semistable with respect to this action.
	The {\it de Rham moduli space} $M_{DR}(\bar{X}/D, P)$ is 
	\[ M_{DR}(\bar{X}/D, P)=R_{DR}(\bar{X}/D, x_0, P)\sslash \text{GL}_r(\mathbb C).\]
	The quotient morphism, denoted
	\[ q_{DR}\colon R_{DR}(\bar{X}/D, x_0, P)\to M_{DR}(\bar{X}/D, P),\]
	is a morphism of $k$-schemes.
	The scheme $M_{DR}(\bar{X}/D, P)$ is a universal categorical quotient,  whose closed points are the
	Jordan equivalence classes of semistable logarithmic connections with Hilbert polynomial $P$, see  \cite[Theorems 4.7, 4.10]{S1}, \cite[Theorem 3.5]{N}. (The condition $LF(\xi)$ from \cite[Theorem 4.10]{S1}  is automatically satisfied in our case.)
The morphism $q_{DR}$ is surjective, and for a  point $( E, \nabla, \xi)\in R_{DR}(\bar{X}/D, x_0, P)(\mathbb C)$, the image
	$ q_{DR}( E, \nabla, \xi)\in M_{DR}(\bar{X}/D, P)(\mathbb C)$
	is the isomorphism class of the logarithmic connection given by the direct sum of the (stable)
	graded quotients with respect to a Jordan-H\"older series of $(E, \nabla)$.
	
	We denote by 
	$$
	M_{DR}(\bar{X}/D,  r)=\coprod_P M_{DR}(\bar{X}/D, P)
	$$
 the disjoint union over Hilbert polynomials $P$ of rank $r$ logarithmic connections (that is, the highest-degree coefficient of $P$ is $r \deg_H(\bar X)/n!$ where $n=\dim X$ and $H$ is the fixed polarization on $\bar X$), and by 
 $$
 q_{DR}\colon R_{DR}(\bar{X}/D, x_0, r)\to M_{DR}(\bar{X}/D, r)
 $$
the associated morphism. This restricts to a morphism
$$
q_{DR}\colon R^{st}_{DR}(\bar{X}/D, x_0, r)\to M^{st}_{DR}(\bar{X}/D, r)
$$
of the open subschemes of $R_{DR}(\bar{X}/D, x_0, r)$ and $M_{DR}(\bar{X}/D, r)$, respectively, whose closed points correspond to stable logarithmic connections, since  $R^{st}_{DR}(\bar{X}/D, x_0, r)$ is also the GIT-stable locus, \cite[Theorem 4.10]{S1}.

\begin{rmk}\label{remHP} $\;$

(i) Recall that to each logarithmic connection $(E,\nabla)$ on $(\bar{X},D)$ one can attach residue maps $$\Gamma_i\in \Hom_{\cO_{D_i}} (E|_{D_i},E|_{D_i})$$
where $D=\cup_{i=1}^sD_i$ is the irreducible decomposition. Moreover, the characteristic polynomial of $\Gamma_i$ is constant along $D_i$. After taking  exponentials of its roots, the resulting polynomial equals  the characteristic polynomial of monodromy around $D_i$ of the local system on $X$ induced by $(E,\nabla)$.

(ii) The Hilbert polynomial of a locally free logarithmic connection $(E,\nabla)$  can be computed in terms of the residue maps $\Gamma_i$. Namely, using Hirzebruch-Riemann-Roch and the formula for the Chern classes of $(E,\nabla)$ from \cite{O}, \cite[(B.3)]{EV}, one has
$$
p(E,t) =
\mathop{\sum_{\al\in \bN^s}}  \frac{(-1)^{|\al|}}{\al_1!\cdot\ldots\cdot\al_s!}\;Tr(\Gamma_1^{\al_1}\circ\ldots\circ\Gamma_s^{\al_s})\int_{\bar{X}}[D_1]^{\al_1}\ldots[D_s]^{\al_s}ch(H)^t td(\bar{X})
$$
for integers $t\ge 0$, where $|\al|=\al_1+\ldots+\al_s$, $Tr$ denotes the trace the composition of residues maps defined on the intersection of all $D_i$ such that $\al_i\neq 0$, $ch(H)=\exp([H])$ is the Chern character of the fixed ample line bundle $H$ on $\bar{X}$, and $td(\bar{X})$ is the Todd class. 

(iii) If $\dim X=1$ every logarithmic connection  is locally free. If the rank $r$ is fixed, fixing the Hilbert polynomial $P=p(E,t)$ is equivalent to fixing the degree of $E$, by the classical Riemann-Roch. The formula from (ii) becomes
$$
p(E,t) = r \deg H\cdot t + r(1-g) -\sum_{i=1}^s Tr(\Gamma_i)
$$
where $g$ denotes the genus of $\bar{X}$. This equality is equivalent to
$
\deg E + \sum_{i=1}^s Tr(\Gamma_i)=0.
$
 \end{rmk}

\subsection{Riemann-Hilbert morphisms}\label{RH}  We recall the construction and some facts about the analytic Riemann-Hilbert maps. This will be useful to give a more concrete criterion for absoluteness.

The Riemann-Hilbert correspondence induces two analytic morphisms between complex analytic spaces, called {\it Riemann-Hilbert morphisms},  forming a commutative diagram
\begin{equation}\label{eqDid}
\begin{gathered}
\xymatrix{
   	R_{DR}(\bar{X}/D, x_0, r)^{an}    \ar[d]_{q_{DR}}  \ar[rr]^{RH} &  & R_{B}(X, x_0, r)^{an}\ar[d]_{q_B}\\
   	 M_{DR}(\bar{X}/D,r)^{an}\ar[rr]^{RH} & & M_B(X, r)^{an}.
}
\end{gathered}
\end{equation}

The construction of $RH$ is due to \cite[7.1, 7.8]{S2} for $X=\bar{X}$. The construction extends to cover the general case as well. In fact $RH$ has been defined more generally from the moduli of regular holonomic $\sD_{\bar{X}}$-modules with singularities along $D$ to a moduli of perverse sheaves with singularities along $D$ in \cite{NS, N2}. We explicit now the construction in our case.

The analytification  $R_{B}(X, x_0, r)^{an}$ is a fine analytic moduli space, representing the  functor  associating to a complex analytic space $T$ the equivalence classes of families over $T$ of framed local systems on $X$, where the latter notions are defined in manner similar to the case of a scheme $T$.  This functor is now isomorphic to the functor associating to $T$ the set of isomorphism classes of pairs $(\cL,\xi)$ where $\cL$ is a  locally free sheaf of $p_2^{-1}\cO_T$-modules of rank $r$ on $X^{an}\times T$, together with an isomorphism $\cL|_{x_0\times T}\xa{\sim}\cO_T^{\oplus r}$, see \cite[Lemma 7.3]{S2} which holds in the quasi-projective case as well.

The analytification  $R_{DR}(\bar{X}/D, x_0, P)^{an}$ is also a fine analytic moduli space, representing the functor associating to a complex analytic space $T$ the equivalence classes of families $(E_T,\nabla_T,\xi_T)$ over $T$ of framed holomorphic semistable logarithmic connections with Hilbert polynomial $P$, where the latter notions are defined in manner similar to the case of a scheme $T$, see \cite[Lemma 5.7]{S1}.

An analytic family $(E_T,\nabla_T,\xi_T)$  of framed holomorphic semistable logarithmic connections gives rise to the analytic family $(\ker (\nabla_T|_{X^{an}\times T}), \xi_T|_{X^{an}\times T})$ of framed local systems. This association, together with the   properties of being fine analytic moduli, gives an analytic morphism $RH:R_{DR}(\bar{X}/D, x_0, r)^{an}\ra R_{B}(X, x_0, r)^{an}$. Since $RH$ is $\GL_r(\bC)$-equivariant and since $M_{DR}(\bar{X}/D,r)^{an}$ and $M_B(X, r)^{an}$ are universal categorical quotients in the category of analytic spaces as well by \cite[Proposition 5.5]{S1}, there is a unique well-defined analytic morphism $RH:M_{DR}(\bar{X}/D,r)^{an}\ra M_B(X, r)^{an}$ making the  diagram (\ref{eqDid}) commute. Pointwise, this last morphism sends the class of a direct sum of stable logarithmic connections on $(\bar{X},D)$ obtained as successive quotients of a Jordan-H\"older series for a semistable logarithmic connection, to the class of the direct sum of the semi-simplifications of the associated local systems on $X$.

Associated to (\ref{eqDid}) is a commutative diagram of  analytic morphisms between the reduced analytic moduli
\begin{equation}\label{eqDidr}
\begin{gathered}
\xymatrix{
   	R_{DR}(\bar{X}/D, x_0, r)(\bC)    \ar@{->>}[d]_{q_{DR}}  \ar[rr]^{RH}  & & R_{B}(X, x_0, r)(\bC)\ar@{->>}[d]_{q_B}\\
   	 M_{DR}(\bar{X}/D,r)(\bC)\ar[rr]^{RH} & & M_B(X, r)(\bC).
}
\end{gathered}
\end{equation} 

The maps $q_{DR}$ and $q_B$ are surjective as we have already seen. The question of surjectivity of the maps $RH$ can be considered as a higher-dimensional version of Hilbert's 21st problem. In this direction, one has the following:

\begin{thm}\label{thmH21}$\;$ 


(1)  The two maps $RH$ from (\ref{eqDidr}) are surjective if:

 $\quad$ (i)  (cf. Lemma \ref{lemDEs}) the rank $r$ is $1$ or, 
 
 $\quad$ (ii) (\cite{S2}) $X$ is projective.
 
(2) (\cite{EVa}) If $\dim X=1$ the  map 
$$RH: M_{DR}(\bar{X}/D, P)(\bC) \to M_{B}(X, r)(\bC)$$ is surjective, where 
$
P(t)=r ((\deg H) t + 1-g)
$
is the Hilbert polynomial of the trivial rank $r$ vector bundle on $\bar{X}$, $H$ is the fixed ample line bundle on $\bar X$, and $g$ is the genus of $\bar X$.

\end{thm}


%


The following questions are open:

\begin{que}  Are the two maps $RH$ in (\ref{eqDidr}) surjective? Or, at least, are their images Zariski constructible?
\end{que}

The main reason behind Theorem \ref{thmH21} is the following: 

\begin{lemma}\label{lemDEs}
 $\ $
 \begin{enumerate}[(i)]
\item (\cite{deligne06}, see also \cite[Ch. 5]{H}) The Deligne extensions of a local system $L$ on $X$ are locally free logarithmic connections giving rise to $L$.
\item (\cite[Proposition 2.3]{N}) Any logarithmic connection on $(\bar{X},D)$ giving rise to a simple local system on $X$ is stable. 
\item (\cite[page 89]{S1}) The category of semistable logarithmic connections of given normalized Hilbert polynomial $p_0$, i.e. $p_0(t)=p(E,t)/rk(E)$, is abelian.
\end{enumerate}
\end{lemma}

A consequence is the surjectivity of RH onto the simple locus.

\begin{defn}
Define $R^s_{DR}(\bar{X}/D, x_0, r)(\bC)$ and $M^s_{DR}(\bar{X}/D,r)(\bC)$ to be the   subsets of $R^{st}_{DR}(\bar{X}/D, x_0, r)(\bC)$ and $M^{st}_{DR}(\bar{X}/D,r)(\bC)$, respectively, corresponding to stable logarithmic connections whose associated local systems on $X$ are simple. That is,
$$R^s_{DR}(\bar{X}/D, x_0, r)(\bC)=  RH^{-1}(R^s_{B}(X, x_0, r)(\bC)),$$
$$M^s_{DR}(\bar{X}/D,r)(\bC) =   RH^{-1}(M^s_B(X, r)(\bC)) = q_{DR} (R^s_{DR}(\bar{X}/D, x_0, r)(\bC)).$$ 
\end{defn}

\begin{lemma}
The subsets $R^s_{DR}(\bar{X}/D, x_0, r)(\bC)$ and $M^s_{DR}(\bar{X}/D,r)(\bC)$ are Zariski open subsets of $R_{DR}(\bar{X}/D, x_0, r)(\bC)$ and $M_{DR}(\bar{X}/D,r)(\bC)$, respectively. 
\end{lemma}
\begin{proof}For a logarithmic connection $(E,\nabla)$, the condition that the associated local system on $X$ is simple is equivalent to the algebraic condition that the restriction of $(E,\nabla)$ to $X$ is irreducible. This is a Zariski closed open condition.  To see this, one can adapt the proof of openness of the (semi)stable locus in a family of logarithmic connections from \cite[Proposition 2.15]{N}, \cite[Lemma 3.7]{S1}. This  shows that $R^s_{DR}(\bar{X}/D, x_0, r)(\bC)$ is Zariski open, by the fine moduli space property of $R_{DR}(\bar{X}/D, x_0, r)(\bC)$. It follows that its image, $M^s_{DR}(\bar{X}/D,r)(\bC)$ is also Zariski open, since $R^s_{DR}(\bar{X}/D, x_0, r)(\bC)$ is $\GL_r(\bC)$-invariant.
\end{proof}

\begin{lemma}\label{lemSj}  The diagram (\ref{eqDidr}) induces a commutative diagram of surjective  maps of complex analytic spaces
\begin{equation}\label{eqDSi}
\begin{gathered}
\xymatrix{
   	R^s_{DR}(\bar{X}/D, x_0, r)^{an}    \ar@{->>}[d]_{q_{DR}}  \ar@{->>}[rr]^{RH} & & R^s_{B}(X, x_0, r)^{an}\ar@{->>}[d]_{q_B}\\
   	 M^s_{DR}(\bar{X}/D,r)^{an}\ar@{->>}[rr]^{RH} &  & M^s_B(X, r)^{an}.
}
\end{gathered}
\end{equation} 
\end{lemma}

The Riemann-Hilbert morphisms are easier to understand when restricted to a smaller  complex analytic open subset.

\begin{defn}
Define $R^{good}_{DR}(\bar{X}/D, x_0, r)(\bC)$ and $M_{DR}^{good}(\bar{X}/D,r)(\bC)$ to be the subsets of $R^s_{DR}(\bar{X}/D, x_0, r)(\bC)$ and $M^s_{DR}(\bar{X}/D,r)(\bC)$, respectively, corresponding to stable locally free logarithmic connections whose associated local systems on $X$ are simple and such that for each irreducible component $D_i$ of $D$ no two eigenvalues of the residue map $\Gamma_i$ differ by a nonzero integer. We call these   {\it good logarithmic connections}. 

It can be easily seen that $R^{good}_{DR}(\bar{X}/D, x_0, r)(\bC)$ and
$M_{DR}^{good}(\bar{X}/D,r)(\bC)$ are open analytic subsets of $R_{DR}(\bar{X}/D, x_0, r)(\bC)$ and $M_{DR}(\bar{X}/D,r)(\bC)$, respectively, obtained as the complement of a locally finite countable union of Zariski closed subsets of $R^s_{DR}(\bar{X}/D, x_0, r)(\bC)$ and $M^s_{DR}(\bar{X}/D,r)(\bC)$, respectively. In particular they underlie complex analytic open subspaces of  $R_{DR}(\bar{X}/D, x_0, r)^{an}$ and $M_{DR}(\bar{X}/D,r)^{an}$, respectively.
\end{defn}

\begin{prop}\label{propGood}$\;$

(1) (\cite[\S 8]{NS}, \cite[7.2]{N2})
The diagram (\ref{eqDid}) induces a commutative diagram of surjective  maps of complex analytic spaces
\begin{equation}\label{eqDSi2}
\begin{gathered}
\xymatrix{
   	R^{good}_{DR}(\bar{X}/D, x_0, r)^{an}    \ar@{->>}[d]_{q_{DR}}  \ar@{->>}[rr]^{RH} & & R^{s}_{B}(X, x_0, r)^{an}\ar@{->>}[d]_{q_B}\\
   	 M^{good}_{DR}(\bar{X}/D,r)^{an}\ar@{->>}[rr]^{RH} &  & M^{s}_B(X, r)^{an}
}
\end{gathered}
\end{equation} 
such that  the horizontal maps $RH$ are local analytic isomorphisms.

(2) (\cite[Theorem 7.1, Proposition 7.8]{S2}) If $X$ is projective, the horizontal maps $RH$ in the diagram (\ref{eqDid}) are analytic isomorphisms.
\end{prop}

\subsection{Moduli-absolute sets}\label{subAbs} Absoluteness for a set of local systems is an abstract categorical notion, and absolute sets are subject to the Special Varieties Package Conjecture. Moduli-absoluteness is a more concrete notion introduced in \cite[Section 7]{BW} via de Rham moduli spaces, in order to  make the conjecture on absolute sets of local systems more approachable. The definition of moduli-absoluteness of \cite[Section 7]{BW} needs a slight correction. While we give in \ref{secApp} a short list of corrections, we go however in this subsection into the details of the comparison between moduli-absoluteness with absoluteness.

    We keep the setup and notation from the previous subsection. We let  $K\subset\bC$ be a field.
     For each $\sigma\in \text{Gal}(\mathbb C/\mathbb Q)$, the base change  $X^\sigma\ra X$   of $X$ via $\sigma$ is compatible with the base changes for $\bar{X}$, $D$, and we denote by  
     \be\label{eqpsi}
     p_\sigma\colon  M_{DR}(\bar{X}/D,r)(\mathbb C)\to  M_{DR}(\bar{X}^\sigma/D^\sigma, r)(\mathbb C)\ee
       the  map induced by base-changing  logarithmic connections. The map $p_\sigma$ is also the base change of the underlying reduced closed $\bC$-subscheme of $M_{DR}(\bar{X}/D,r)$ via $\sigma$, and hence it is also a homeomorphism for the Zariski topology. In particular $p_\sigma$ is bijective. We consider the diagram of maps
 \be\label{eqDrH}
\begin{gathered}
\xymatrix{
   	M_{DR}(\bar{X}/D,r)(\mathbb C)   \ar[d]_{RH}  \ar[rr]^{p_\sigma} & & M_{DR}(\bar{X}^\sigma/D^\sigma,r)(\mathbb C)\ar[d]_{RH}\\
   	 M_B(X,r)(\bC)&  &  M_B(X^\sigma,r)(\bC).
}
\end{gathered}
\ee

The following is clear:

\begin{lemma}\label{lemSame}
If $S\subset M_B(X,r)(\bC)$ is contained in the image of the analytic $RH$ map,  $\sigma\in Gal(\bC/\bQ)$, and $S^\sigma$ is defined as in (\ref{eqSss}), then $S^\sigma$ also equals ${RH}\circ p_\sigma\circ {RH}^{-1}(S)$ with $RH, p_\sigma$ as in (\ref{eqDrH}).
\end{lemma}

\begin{definition}\label{defAbs}  $\;$

(1) A {\it moduli-absolute $K$-constructible (respectively, moduli-absolute $K$-closed) subset of $ M_{B}(X,r)(\mathbb C)$} is a subset $S$ contained in the image of the analytic $RH$ map from (\ref{eqDrH}) and such that for every $\sigma\in \text{Gal}(\mathbb C/\mathbb Q)$,
	  $S^\sigma$
  is $K$-constructible  (respectively,   $K$-closed) in $M_{B}(X^\sigma, r)(\mathbb C)$.     

(2) More generally, if $M\subset M_B(X,r)(\bC)$ is absolute $K$-constructible, a {\it moduli-absolute $K$-constructible (respectively, moduli-absolute $K$-closed) subset of $M$} is a subset $S$ contained in the intersection of the image of $RH$ with $M$, such that for every $\sigma\in \text{Gal}(\mathbb C/\mathbb Q)$, 
	  $S^\sigma$
is $K$-constructible  (respectively,   $K$-closed) in $M^\sigma$.  
\end{definition}

\begin{rmk}$\;$

(1) The first definition is  a slightly adjusted version of \cite[Definition 7.4.1]{BW} which needed a correction, cf.  \ref{secApp}. 

(2) If in the second part of the definition $M$ lies in $M_B^s(X,r)(\bC)$, then we know that it lies in the image of $RH$ already.  

(3) If the analytic map $RH$ is not surjective onto $M_B(X,r)(\bC)$, then $M_{DR}(\bar X/D,r)$ is not the right moduli to capture all absolute subsets of $M_B(X,r)$. 
\end{rmk}

\begin{prop}\label{propBiAlg}$\;$ 

(1) If $S\subset M_B(X,r)(\bC)$ is contained in the image of the analytic $RH$ map from (\ref{eqDrH}), then: $S$ is moduli-absolute $K$-constructible (respectively, moduli-absolute $K$-closed)   iff it is  absolute $K$-constructible (respectively, absolute $K$-closed). 

(2) More generally, if $M\subset M_B(X,r)(\bC)$ is absolute $K$-constructible, and if $S\subset M$ is contained in the image of the analytic $RH$ map from (\ref{eqDrH}), then: $S$ is moduli-absolute $K$-constructible (respectively, moduli-absolute $K$-closed) in $M$  iff it is  absolute $K$-constructible (respectively, absolute $K$-closed) in $M$.

\end{prop}
\begin{proof}  This follows directly from Proposition  \ref{lemAnlg} and Lemma \ref{lemSame}.
\end{proof}    
    
\begin{rmk} $\;$

(1) If the analytic $RH$ map from (\ref{eqDrH}) and its version from $R_{DR}(\bar{X}/D,x_0,r)(\bC)$ to $R_B(X,x_0,r)(\bC)$ are both surjective, part (1)  was proved in  \cite[Proposition 7.4.4]{BW}, if the correction as in cf. \ref{secApp} is taken into account.

(2) At least in the  case $M\subset M_B^s(X,r)(\bC)$, the definition of moduli-absolute subsets of $M$ does not depend on the choice of good compactification $(\bar{X},D)$ of $X$ used to define the de Rham moduli space. 

\end{rmk}

The analytic $RH$ map can have non-trivial fibers. However, since the Riemann-Hilbert correspondence is an equivalence of categories, the fibers satisfy some constraints:

\begin{lemma}\label{lemFib} For all subsets $S\subset M_B(X,r)(\bC)$ contained in the image of the analytic $RH$ map, and all $\sigma\in  Gal (\mathbb C/\mathbb Q)$, the following hold:
\begin{enumerate}
\item  $p_\sigma$ sends any fiber of $RH: M_{DR}(\bar{X}/D,r)(\mathbb C)\to M_B(X,r)(\mathbb C)$ bijectively onto a fiber of $RH: M_{DR}(\bar{X}^\sigma/D^\sigma,r)(\mathbb C)\to M_B(X^\sigma,r)(\mathbb C)$. 
\item $RH^{-1}(S^\sigma)=p_\sigma(RH^{-1}(S))$.
\end{enumerate}
\end{lemma}

\begin{rmk} The analog of Lemma \ref{lemOps} holds for moduli-absolute sets in $M$, for all absolute $K$-constructible subsets $M$ of $M_B(X)(\bC)$.
\end{rmk}

Moduli-absoluteness is closely related to bi-algebraicity:

\begin{proposition}\label{propBia} (\cite[Proposition 7.4.5]{BW}, cf. \ref{secApp}) Let $X$ be a smooth complex quasi-projective algebraic variety defined over a countable subfield $K$ of $\bC$. Let $\bar{X}$, $D$, $x_0$ be as above and assume all of them are defined over $K$.  Let $S$ be a moduli-absolute $K$-constructible subset of $M_B(X,r)(\bC)$. Then the Euclidean closure of any analytic irreducible component of $RH^{-1}(S)$ is a Zariski closed subset of $M_{DR}(\bar{X}/D,r)(\bC)$.
\end{proposition}

\begin{rmk}\label{propBia2}
 If $S$ is as above, there could be infinitely, but at most countable, many analytic irreducible components $T$ of $RH^{-1}(S)$. 
\end{rmk}

Let us spell out what this proposition implies for absolute closed subsets.

\begin{lemma}\label{lemko}
Let $X$, $\bar{X}$, $D$, $x_0$ be as above, all of them defined over a countable subfield $K$ of $\bC$. Let $S$  be a moduli-absolute $K$-closed subset of $M_B(X,r)(\bC)$. Let $S=\cup_{i\in I}S_i$ be the decomposition into Zariski irreducible components, with $I$ a finite set. Let $RH^{-1}(S)=\cup_{j\in J}T_j$ be the decomposition into analytic irreducible components, where $J$ can be infinite. Then:
\begin{enumerate}
\item  For all $j\in J$, $T_j$ is Zariski closed and is an analytic irreducible component of $RH^{-1}(S_i)$ for some $i\in I$ depending on $j$.
\item Conversely, if $T'$ is an analytic irreducible component of $RH^{-1}(S_i)$ for some $i\in I$, such that $RH(T')$ is Zariski dense in $S_i$ (such components always exists for every $i$), then $T'=T_j$ for some $j\in J$ depending on $T'$.  
\item If $\sigma\in  Gal (\bC/\bQ)$, then
   $p_\sigma$ induces a bijection between the irreducible components of $RH^{-1}(S)$ and those of  $RH^{-1}(S^\sigma)=p_\sigma(RH^{-1}(S))$.  
\end{enumerate}
\end{lemma}
\begin{proof}
(1) is a direct application of Proposition \ref{propBia}. (2) holds since $RH$ is an analytic map. (3) follows from the first assertion in (1) together with the fact that $p_\sigma$ is a homeomorphism for the Zariski topology.
\end{proof}

The first two assertions say that every analytic irreducible component of $RH^{-1}(S_i)$ which is not algebraic is contained in some irreducible component, necessarily algebraic, of $RH^{-1}(S)$.

If $K=\bbQ$ it is expected that  $S^\sigma=\cup_{i\in I}S_i^\sigma$ is the Zariski irreducible decomposition of $S^\sigma$, see Question \ref{queIrr}, cf. Conjecture \ref{conj1}.  Even showing that $S_i^\sigma$ are constructible is at the moment open. For this, it would be useful to know if the analytic morphism $RH$ is special enough to allow one to write each $S_i$ in terms of the components $T_j$.

\subsection{The submoduli of simple local systems}\label{subslli}

We keep the setup from Section \ref{defderham}.

\begin{prop}\label{propZ2} $\;$

(1) For $\sigma\in  Gal (\mathbb C/\mathbb Q)$, 
$$
(M_B^s(X,r)(\bC))^\sigma = M_B^s(X^\sigma,r)(\bC),
$$ 
$$
p_\sigma(M_{DR}^s(\bar{X}/D,r)(\bC)) = M_{DR}^s(\bar{X}^\sigma/D^\sigma,r)(\bC),$$ $$p_\sigma(M_{DR}^{good}(\bar{X}/D,r)(\bC)) = M_{DR}^{good}(\bar{X}^\sigma/D^\sigma,r)(\bC).
$$

In particular, $M_B^s(X,r)(\bC)$ is absolute $\bQ$-constructible in $M_B(X,r)(\bC)$. Its complement is absolute $\bQ$-closed.

(2) $M_B^s(X,r)(\bC)$  is moduli-absolute $\bQ$-constructible in $M_B(X,r)(\bC)$.

(3) If $S_i$ with $i\in I$ denote the irreducible components of $M_B^s(X,r)(\bC)$, then $S_i^\sigma$ are the irreducible components of $M_B^s(X^\sigma,r)(\bC)$ for all $\sigma\in Gal(\bC/\bQ)$. Hence $S_i$ are (moduli-) absolute $\bbQ$-constructible.

\end{prop}
\begin{proof} (1) The Riemann-Hilbert correspondence between regular holonomic $\sD_X$-modules and perverse sheaves on $X$ preserves simplicity. This implies the first equality. The second equality follows from the first together with Lemma \ref{lemFib} (2) since $M_{DR}^s(\bar{X}/D,r)(\bC)=RH^{-1}(M^s_B(X,r)(\bC))$. The second equality implies the third since $p_\sigma$ preserves the good property for logarithmic connections.

The last claim follows from the first equality since  $M_B^s(X^\sigma,r)(\bC)$ is  a Zariski open subset of $M_B(X^\sigma,r)(\bC)$ defined over $\bQ$.

(2) By Lemma \ref{lemSj}, $M_B^s(X,r)(\bC)$ lies in the image of the analytic $RH$ map.

(3) This will follow from Theorem \ref{propB1Ms} below.
\end{proof}

\begin{rmk}
Since there is no difference between moduli-absolute and absolute subsets of $M_B^s(X,r)(\bC)$, we  stop using the terminology ``moduli-absolute" for such sets. 
\end{rmk}

For the rest of this subsection the goal is to prove Theorem \ref{propA1si}. Before we can give the proof, we need a few lemmas. Denote for simplicity for the rest of this section $$M_{\overline{X}/D}\coloneqq  M_{DR}^{good}(\overline{X}/D, r)(\bC).$$ Recall that the restriction of the analytic $RH$ map $$ RH_0\colon M_{\overline{X}/D}\twoheadrightarrow M_B^s(X, r)(\bC)$$
is a surjective analytic local isomorphism by Proposition \ref{propGood}.

\begin{lemma}\label{lemRH0}
	
	Let $S $  be a Zariski constructible subset of $ M_B^s(X,r)(\bC)$. Let $\overline{S}$ be the Zariski closure (or equivalently, the Euclidean closure) of $S$ in $M_B^s(X,r)(\bC)$. Let $A=RH^{-1}(S)$ be the inverse image, and  $\overline{A}$ be the Euclidean closure  of $A$ in $M_{DR}(\overline{X}/D, r)(\bC)$. Then $RH_0^{-1}(\overline{S})=\overline{A}\cap M_{\overline{X}/D}$.
\end{lemma}

\begin{proof}
	The inverse image $RH^{-1}(\overline{S})$ of $\overline{S}$ is an analytic closed  set containing $A$, so $\overline{A}\subset RH^{-1}(\overline{S})$ and  $$\overline{A}\cap M_{\overline{X}/D}\subset RH^{-1}(\overline{S})\cap M_{\overline{X}/D}=RH_0^{-1}(\overline{S}).$$
	In particular,  $S\subset  RH(\overline{A}\cap M_{\overline{X}/D})\subset \overline{S}$.
	If $RH_0^{-1}(\overline{S})\neq \overline{A}\cap M_{\overline{X}/D}$, then we can find a point $P\in RH_0^{-1}(\overline{S}) \setminus (\overline{A}\cap M_{\overline{X}/D})$. Put $Q=RH(P)$, and then $Q\in \overline{S}\setminus S$. Since $RH_0$ is a surjective analytic local  isomorphism, there exists an open neighborhood $U_P\subset M_{\overline{X}/D}$ of $P$ with respect to the Euclidean topology,  such that $U_P\cap (\overline {A}\cap M_{\overline{X}/D})=\emptyset$ and $RH_0\colon U_P \to V_Q=RH_0(U_P)$ is an analytic isomorphism.
	Since $U_P\cap (\overline {A}\cap M_{\overline{X}/D})=\emptyset$, we have $V_Q\cap S=\emptyset$. This is impossible, because $Q$ is a point in $\overline{S}$.  Hence, $RH_0^{-1}(\overline S)=\overline A\cap M_{\overline{X}/D}$. 
\end{proof}

Next, recall that an {\it analytically constructible subset} of an analytic space is an element of the smallest family of subsets  that contains all analytic subsets and is closed with respect to the operations of taking the finite union of sets and the complement of a set. Since every analytic space can be represented uniquely as a union of a locally finite family of irreducible analytic subspaces, the family of analytic irreducible components of an analytically constructible set is also locally finite.

\begin{lemma}\label{lemclosureofdecom}
	Let $A$ be an analytically constructible subset of an analytic space $M$, and let $A= \cup_{j\in J} A_j$ be the decomposition into analytic irreducible components. Then the Euclidean closure $\overline{A}$ of $A$  satisfies $$\overline{A} = \cup_{j\in J} \overline{A_j},$$
	where $\overline{A_j}$ is the Euclidean closure of $A_j$ for every $j\in J$.
\end{lemma}
\begin{proof}
	The set $\overline{A}$ is an Euclidean closed set containing every  $A_j$,   so $\overline{A_j}\subset \overline{A}$ for all $j\in J$. Thus $\cup_{j\in J}\overline{A_j}\subset \overline{A}$. 
	Since $A$ is an analytically constructible subset  of $M$, the union $\cup_{j\in J}\overline{A_j}$  is  locally a finite union, and hence it is analytic. Thus the set $\cup_{j\in J}\overline{A_j}$ is an  Euclidean closed subset containing $A$ and so $\overline{A}\subset \cup_{j\in J}\overline{A_j} $. Therefore, $\overline{A}= \cup_{j\in J}\overline{A_j}$ as we expected.
\end{proof}

\begin{lemma}\label{lemAICsigma} Let  $S $  be an absolute $\overline{\bQ}$-constructible subset of $ M_B^s(X,r)(\bC)$. Let $\sigma\in \text{Gal}(\bC/\bQ)$ and $S^\sigma = RH\circ p_\sigma \circ RH^{-1}(S).$   Let 
$$A=RH^{-1}(S)=\cup_{j\in J}A_j \quad \text{and}\quad B=RH^{-1}(S^\sigma)=\cup_{k\in K}B_{k}
$$
be the decompositions into analytic irreducible components in $M_{DR}(\overline{X}/D, r)(\bC)$.  Let $\overline{A_j}$, $\overline{B_k}$ be the Euclidean closures in $M_{DR}(\overline{X}/D, r)(\bC)$ and $M_{DR}(\overline{X}^\sigma/D^\sigma, r)(\bC)$, respectively. Then:

(1) For every $j\in J$, we have that $$p_\sigma(\overline{A_j})=\overline{B_{k(j)}},$$ for some $k(j)\in K$ depending only on $j$.

(2) The map $p_\sigma$ induces a bijection between the set $\{\overline{A_j}\}$ of Euclidean closures of analytic irreducible components of $A$ and the set $\{\overline{B_k}\}$ of Euclidean closures of analytic irreducible components of $B$.  

(3) If $\overline{A}$, $\overline{B}$ denote the Euclidean closures, then $\overline{B}= p_\sigma(\overline{A})$.
\end{lemma}

\begin{proof} The set $S^\sigma$ is also an absolute $\overline{\bQ}$-constructible subset of $M_B(X^\sigma, r)(\bC)$, by Proposition \ref{propBiAlg}. By the absoluteness of $S$, $S^\sigma$, and Proposition \ref{propBia}, $\overline{A_j}$ and $\overline{B_k}$ are irreducible Zariski closed sets for every $j\in J$, $k\in K$. The map $p_\sigma$ and its inverse $p_\sigma^{-1}=p_{\sigma^{-1}}$ preserve Zariski closed subsets and irreducibility with respect to the Zariski topology. 
In particular, $p_\sigma(\overline{A_j})$, $p_\sigma^{-1}(\overline{B_k})$ are irreducible Zariski closed sets  for every $j\in J$, $k\in K$.  Since $S^\sigma$ is also absolute,  (2) follows from (1) by interchanging the roles of $S$ and $S^\sigma$. Then (3) follows from (2) by applying Lemma \ref{lemclosureofdecom}. Thus it remains to prove (1).

By Lemma \ref{lemFib},  $B=p_\sigma(A)$, so the following holds
	$$\cup_{j\in J} A_j=
	A=p_\sigma^{-1}(B)\subset p_\sigma^{-1}(\overline{B})=\cup_{k\in K} p_\sigma^{-1}(\overline{B_k}).$$
The component $A_j$ is analytic irreducible, so $A_j\subset p_\sigma^{-1}(\overline{B_{k(j)}})$ for some $k(j)\in K$ depending only on $j$, and for all $j\in J$.
	Fix  $j\in J$.  As $p_\sigma^{-1}(\overline{B_{k(j)}})$ is a  Zariski closed set containing $A_j$, we get that $$\overline{A_j}\subset 	p_\sigma^{-1}(\overline{B_{k(j)}}).$$
	Acting on both sides by $p_\sigma$ thus gives \begin{equation}\label{eq1.2}  p_\sigma(\overline{A_j})\subset 	p_\sigma\circ p_\sigma^{-1}(\overline{B_{k(j)}})=\overline{B_{k(j)}}.\end{equation}

	We also show that $p_\sigma^{-1}(\overline{B_{k(j)}})\subset \overline{A_j}$ so that  $p_\sigma^{-1}(\overline{B_{k(j)}})=\overline{A_j}$.
	Denote $k= k(j)$ for simplicity. Similarly as before, there exists  $\ell(k)\in J$ depending only on $k$ such that  $p_\sigma^{-1}(\overline{B_{k}})=p_{\sigma^{-1}}(\overline{B_{k}})\subset \overline{A_{\ell(k)}}. $
	Together with (\ref{eq1.2})  this  implies that $$\overline{A_j}=p_\sigma^{-1}\circ p_\sigma (\overline{A_j})\subset  p_\sigma^{-1}(\overline{B_{k}})\subset \overline{A_{\ell(k)}}.$$
	As $\overline{A_j}$ and $\overline{A_{\ell(k)}}$  are the Euclidean closures in $M_{DR}(X/D, r)(\bC)$ of two  analytic irreducible components $A_j$ respectively $A_{\ell(k)}$ of $A$,  we have that $A_j=A\cap \overline{A_j}\subset A\cap \overline{A_{\ell(k)}}=A_{\ell(k)}.$
	The uniqueness of decomposition into  irreducible components thus tells us that $A_j=A_{\ell(k)}$, and so $\overline{A_j}\subset p_\sigma^{-1}(\overline{B_{k(j)}})=p_\sigma^{-1}(\overline{B_{k}}) \subset \overline{A_j}.$
	It further implies, by taking image via $p_\sigma$,  that $p_\sigma(\overline{A_j})\subset \overline{B_{k(j)}}\subset  p_\sigma(\overline{A_j}).$
	Therefore,   $  p_\sigma(\overline{A_j})=\overline{B_{k(j)}}.$ As an immediate consequence, we also have an induced bijection between $J$ and $K$, as claimed. 
\end{proof}

\begin{lemma}\label{corclosure}
	The Euclidean (or equivalently, the Zariski) closure $\overline{S}$ in $M_B^s(X,r)(\bC)$ of an absolute $\bar{\bQ}$-constructible subset $S$ is an absolute $\bar{\bQ}$-closed subset in $M_B^s(X,r)(\bC)$.
\end{lemma}
\begin{proof}	We keep the notation from Lemma \ref{lemRH0} and Lemma \ref{lemAICsigma}. Let $\sigma\in \text{Gal}(\bC/\bQ)$.
	By the previous lemma,  we have $\overline{B}=p_\sigma(\overline{ A})$ together with an identification of decompositions into analytic irreducible components $$\overline{B}=\cup_{k\in K}\overline{B_k}=\cup_{j\in J}p_\sigma(\overline{A_j})=p_\sigma(\overline{A}).$$
	By Lemma \ref{lemRH0}, 	$RH_0^{-1}(\overline{S})=\overline{A}\cap M_{\overline{X}/D}$.
Since  $RH_0$ is surjective, one also has \begin{equation}\label{eq1.1}S^\sigma= RH_0\circ p_\sigma\circ RH_0^{-1}(S).\end{equation}
By Lemma \ref{lemRH0} applied to $S^\sigma$, it then follows that  $$RH_0^{-1}(\overline{S^\sigma})= \overline{B}\cap M_{\overline{X}^\sigma/D^\sigma}=p_\sigma(\overline{A})\cap p_\sigma(M_{\overline{X}/D})=p_\sigma(RH_0^{-1}(\overline{S})).$$
	Together with the surjectivity of $RH_0$ it implies that  $\overline{S}$  is  an absolute $\overline{\bQ}$-closed subset of $M_B^s(X,r)(\bC)$.
\end{proof}

\begin{lemma}\label{lemwq}
Every absolute $\bbQ$-constructible subset $S$ in $M_B^s(X,r)(\bC)$ is a Boolean combination of absolute $\bbQ$-closed subsets in $M_B^s(X,r)(\bC)$.
\end{lemma}
\begin{proof}
 For a subset of dimension 0 this is clear since the assignment $P\mapsto P^\sigma=RH\circ p_\sigma\circ RH^{-1}(P)$ is a bijection $ M_B^s(X,r)(\bC)\to M_B^s(X^\sigma,r)(\bC)$. If $\dim S>0$ consider the complement $T$ of $S$ in $\overline{S}$. It is a Zariski constructible subset  defined over $\overline{\bQ}$ in $M_B^s(X,r)(\bC)$ and $T^\sigma= \overline{S^\sigma}\setminus S^\sigma.$ Thus the set $T^\sigma$ is  a Zariski constructible subset defined over $\bar{\bQ}$, so $T$ is absolute $\overline{\bQ}$-constructible in $M_B^s(X, r)(\bC)$. Moreover $\dim T < \dim \overline S=\dim S.$ By induction on dimension,  $T$ is a Boolean combination of absolute closed $\bbQ$-subsets in $M_B^s(X,r)(\bC)$. Hence the same holds for $S$.
\end{proof}

Denote by $Sing(\_)$ the singular locus of a  complex analytic space or of a $\bC$-scheme locally of finite type. We will need the following, essentially suggested to us by the referee:

\begin{lemma}\label{lemSinga}
Let $M$ be a $\bC$-scheme locally of finite type. Let $N\subset M(\bC)$ be the complement of a locally finite countable union of Zariski closed subsets. Let $\sigma\in Gal(\bC/\bQ)$ and let $p_\sigma:M\to M^\sigma$ be the morphism of schemes induced by base change along $\sigma$. Let $A\subset M(\bC)$ be an analytically constructible subset such that the Euclidean closure of every irreducible component of $A$ is Zariski closed. Let $\ol{A}$ be the Euclidean closure of $A$ in $M(\bC)$. Then $p_\sigma(Sing(\ol{A}\cap N))=Sing(p_\sigma(\ol{A}\cap N))$.
\end{lemma}
\begin{proof}
Let us first check that both sides of the claimed equality are well-defined. 

If $A=\cup_{j\in J}A_j$ is the decomposition into analytic irreducible components, then $\ol{A}=\cup_{j\in J}\ol{A_j}$, where $\ol{A_j}$ is the Euclidean closure, by Lemma \ref{lemclosureofdecom}. By assumption $\ol{A_j}$ is  Zariski closed. Since $J$ is locally finite it follows that $\ol{A}\cap N$ is a complex analytic space and hence its singular locus is well-defined. 

Let $\cup_{k\in K}N_k$ be the complement of $N$, which by assumption is a locally finite countable union of Zariski closed subsets of $M(\bC)$. Then $\ol{A}\cap N =\cup_j\cap_k \ol{A_j}\setminus N_k$. Since $p_\sigma$ is a homeomorphism of Zariski topologies, it follows that $p_\sigma(\ol{A}\cap N)=\cup_j\cap_k p_\sigma(\ol{A_j})\setminus p_\sigma(N_k)$, with $p_\sigma(\ol{A_j})$ and $p_\sigma(N_k)$ Zariski closed. Thus $p_\sigma(\ol{A}\cap N)$ is also a complex analytic space and its singular locus is well-defined.

Let $x$ be a point of $\ol{A}\cap N$. Then there exists a finite subset $J(x)\subset J$ and a small Euclidean open neighborhood $B(x)$ of $x$ in $M(\bC)$ such that $\ol{A}\cap N\cap B(x)=\cup_{j\in J(x)}\ol{A_j}\cap B(x)$. Hence $x$ is a singular point of $\ol{A}\cap N$ if and only if it is a singular point of the $\bC$-scheme $Z(x)=\cup_{j\in J(x)}\ol{A_j}$. 
Since $M$ is locally of finite type, $Sing(Z(x))$ is Zariski closed. Since conjugation by $\sigma$ preserves algebraic singular loci, $p_\sigma(Sing(Z(x)))=Sing(p_\sigma(Z(x)))=Sing(\cup_{j\in J(x)}p_\sigma(\ol{A_j}))$. Moreover, $\cup_{j\in J(x)}p_\sigma(\ol{A_j})$ agrees with $p_\sigma(\ol{A}\cap N)$ on a small Euclidean open neighborhood of $p_\sigma(x)$. 
This implies the claimed equality.
\end{proof}

We will apply this to obtain:

\begin{lemma}\label{lemirredcom}
	If $S$ is absolute $\bbQ$-closed in $M_B^s(X,r)(\bC)$, then every Zariski irreducible component of $S$ is also absolute $\bbQ$-closed. 
\end{lemma}
\begin{proof}  Since $RH_0$ is an analytic local isomorphism, $Sing(RH_0^{-1}(S)) = RH_0^{-1}(Sing(S))$. By Lemma \ref{lemRH0} we have $RH_0^{-1}(S)=\ol{A}\cap M_{\overline{X}/D}$ with $A=RH^{-1}(S)$. We apply Lemma \ref{lemSinga} with $M=M_{DR}(\bar X/D,r)(\bC)$, $N=M_{\overline{X}/D}$. Recall that $M_{\overline{X}/D}$ is the complement of a locally finite countable union of Zariski closed subsets of $M_{DR}(\bar X/D,r)(\bC)$, and that 
 by Proposition \ref{propBia} the Euclidean closure of all analytic irreducible component of $A$ are Zariski closed. Hence Lemma \ref{lemSinga} can be applied. The consequence is that $p_\sigma(Sing(RH_0^{-1}(S))) = Sing(p_\sigma(RH_0^{-1}(S)))$ for all $\sigma\in Gal(\bC/\bQ)$. Since $Sing(p_\sigma(RH_0^{-1}(S)))=RH_0^{-1}(Sing(S^\sigma))$, we have obtained that $(Sing(S))^\sigma=Sing(S^\sigma)$ and so $Sing(S)$ is absolute $\bbQ$-locally closed in $M_B^s(X,r)(\bC)$. Thus the smooth locus $S_{sm}=S\setminus Sing(S)$ is absolute $\bbQ$-locally closed in $M_B^s(X,r)(\bC)$, and smoothness is an absolute notion, that is $S_{sm}^\sigma=(S^\sigma)_{sm}$. There is a bijection between the analytic  irreducible components of $S$ (and therefore Zariski irreducible components since $S$ is algebraic) and the  connected components of $S_{sm}$, see \cite[II.5.3]{De}. The bijection is given by taking the closure of the latter in $S$, or equivalently in $M_B^s(X,r)(\bC)$ since $S$ is closed. Hence it suffices to show that if $S$ is smooth absolute $\bbQ$-locally closed in $M_B^s(X,r)(\bC)$, then every connected component of $S$ is also absolute $\bbQ$-locally closed.

	Let $S$ be smooth absolute $\bbQ$-locally closed in $M_B^s(X,r)(\bC)$. Let $S_j$ with $j\in J$ be the connected components of $S$. Let $T_h$ with $h\in H$ be the connected components of $RH_0^{-1}(S)$. Using the almost-algebraicity of $T_h$ as above, we have that $p_\sigma(T_h)$ with $h\in H$ are the connected components of $RH_0^{-1}(S^\sigma)$.  Thus the restriction of $RH_0$ gives  analytic local isomorphisms $T_h\to S$ and $p_\sigma(T_h)\to S^\sigma$, in particular these are open maps in the Euclidean topology. Let $H(j)=\{h\in H\mid RH_0(T_h)\subset S_j\}$. Then $\cup_{h\in H(j)}RH_0(T_h)=S_j$. Moreover, $RH_0(p_\sigma(T_h))$ is an analytic open subset of $S^\sigma$. Thus $S_j^\sigma=\cup_{h\in H(j)}RH_0(p_\sigma(T_h))$ is an analytic open subset of $S^\sigma$ for every $\sigma$. 
	
	
	To conclude that $S_j$ is absolute $\bbQ$-locally closed, it is enough to show that $S_j^\sigma$ is also Zariski open. It is thus enough to prove that $S_j^{\sigma} := \cup_{h\in H(j)} RH_0(p_{\sigma}(T_h))$ is a Euclidean connected component of $S^\sigma$, since Zariski and Euclidean connectedness are equivalent properties. First, we prove it is connected. Pick two points $x,y\in S_j^{\sigma}$ and let $x',y'\in S_j$ be the corresponding points in $S$. By considering a covering of a path between $x',y'\in S_j$ by opens $RH_0(T_h)$, one can see that there is a sequence $h_1,\ldots, h_s\in H(j)$ such that $x'\in RH_0(T_{h_1})$, $y'\in RH_0(T_{h_s})$ and with $RH_0(T_{h_{l}})\cap RH_0(T_{h_{l+1}})\neq\emptyset$, for $l=1,\ldots, s-1$. Applying $p_\sigma$ and noting that $p_\sigma$ sends fibers of $RH_0$ to fibers of $RH_0$, we have that $x\in RH_0(p_{\sigma}(T_{h_1}))$, $y\in RH_0(p_{\sigma}(T_{h_s}))$ and $RH_0(p_{\sigma}(T_{h_{l}}))\cap RH_0(p_{\sigma}(T_{h_{l+1}}))\neq\emptyset$,  for $l=1,\ldots, s-1$. It follows that $x,y$ lie in the same connected component of $S^\sigma$. 
	
	On the other hand, if $S_j^{\sigma}$ would not be a connected component there would exist an index $j'\in J$, $j'\neq j$ and indices $h'\in H(j')$, $h\in H(j)$ such that $RH_0( p_{\sigma}(T_{h})) \cap RH_0( p_{\sigma}(T_{h'}))\neq \emptyset$. But then we would have $RH_0(T_{h}) \cap RH_0(T_{h'})\neq \emptyset$ which is a contradiction since $RH_0(T_{h})$ and $RH_0(T_{h'})$ lie in different connected components.
\end{proof}

\begin{thm}\label{propB1Ms}{\bf ($=$ Theorem \ref{propA1si}.)}
Conjecture A1 holds in $M_B^s(X,r)(\bC)$. That is: the Euclidean (or equivalently, the Zariski) closure in $M_B^s(X,r)(\bC)$ of an absolute $\bar{\bQ}$-constructible set is an absolute $\bar{\bQ}$-closed subset of $M_B^s(X,r)(\bC)$; the collection of absolute $\bar{\bQ}$-constructible subsets of $M_B^s(X,r)(\bC)$ is generated from the absolute $\bar{\bQ}$-closed subsets of $M_B^s(X,r)(\bC)$ via finite sequences of taking unions, intersections, and complements; and, every irreducible component of an absolute $\bbQ$-closed subset in $M_B^s(X,r)(\bC)$ is absolute $\bbQ$-closed.
\end{thm}
\begin{proof}
The claims were proved in Lemma \ref{corclosure}, Lemma \ref{lemwq}, and Lemma \ref{lemirredcom}.	
\end{proof}

\subsection{Monodromy at infinity and determinant}\label{subMinf} We keep the notation of this section, namely  $X$ is a smooth complex quasi-projective variety, $\bar{X}$ a smooth compactification, the complement $D$  is a simple normal crossing divisor. Let $D=\cup_{i=1}^s D_i$ be the irreducible decomposition.  Denote by $M_{DR}(\overline{X}/D,r)^\text{lf}(\bC)$ subset of $M_{DR}(\overline{X}/D,r)(\bC)$ consisting of polystable logarithmic connections such that each stable direct summand is a locally free $\cO_{\bar{X}}$-module. 
Then $M_{DR}(\overline{X}/D,r)^\text{lf}(\bC)$ is the set of closed points of a Zariski open subscheme of $M_{DR}(\overline{X}/D,r)$, by \cite[Proposition 5.21]{N2}. Consider the diagram:
\be\label{eqDlf}
\begin{gathered}
\begin{tikzcd}
{M_{DR}(\overline{X}/D,r)}^\text{lf}(\bC) \arrow[d, "RH"] \arrow[r, "res"] &  \mathbb{C}^{rs} \arrow[d, "f"] &  \mathbb{C}^{rs} \arrow[d, "Exp"] \arrow[l, "root"] \\
{M_B(X,r)(\bC)} \arrow[r, "mon"]                           &   (\mathbb{C}^{r-1}\times\bC^*)^{ s} = ((\bC^*)^{(r)})^s         &  (\mathbb{C}^*)^{ rs} \arrow[l, "root"]                  
\end{tikzcd}
\end{gathered}
\ee
defined as follows. Here $res = (res_1,\ldots, res_s)$ and  $res_i:M_{DR}(\overline{X}/D,r)^\text{lf}(\bC)\to \mathbb{C}^r$ is the algebraic map which associates to $(E,\nabla)$ the $r$-tuple of coefficients, excluding the top degree, of the characteristic polynomial of the residue homomorphism $\Gamma_i\in \Hom(E|_{D_i}, E|_{D_i})$. This is well-defined, since over a point of $D_i$ the residue homomorphism determines a conjugacy class in $\GL_r(\mathbb{C})$ which remains constant as the point varies. Similarly, $mon=(mon_1,\ldots, mon_s)$ and $mon_i: M_B(X,r)(\bC)\to \mathbb{C}^{r-1}\times\bC^*$ is the morphism defined over $\bbQ$ that associates to a representation $\rho:\pi_1(X,x_0)\to \GL_r(\mathbb{C})$ the coefficients of the characteristic polynomial of the matrix associated to a small loop around $D_i$. The constant term coefficient is nonzero. 
By identifying a monic polynomial with the unordered set of its roots, one identifies $\mathbb{C}^{r-1}\times\bC^*$ 
with  the $r$-th symmetric product $(\bC^*)^r/S_r$ denoted $(\bC^*)^{(r)}$.  Each of two maps $root$ is the product $s$ times of the quotient maps by the action of the symmetric group $S_r$. The quotient map is a finite map of algebraic varieties, given by the symmetric polynomials, such that the fiber over $(a_{0},\ldots, a_{r-1})$ is the set of  roots of the polynomial $t^r+a_{r-1}t^{r-1}+\ldots + a_0$ and their distinct permutations. The map $Exp$ takes $z$ to  $e^{2\pi iz}$ component-wise.   The map $f$ is the unique map that makes the diagram commutative.

\begin{lemma}\label{lemMinf} Let $M$ be one of the following:

(a) $M=M_B(X,r)$, in which case we also assume that $RH:M_{DR}(\overline{X}/D,r)^\text{lf}(\bC)\to M_B(X,r)(\bC)$ is surjective; or,

(b) $M=M_B^s(X,r)$, in which case $mon$ will denote the restriction $mon|_M$.

We have:

(1) If $S\subset M$ is an absolute $\bbQ$-constructible subset, then
$
root^{-1}(mon(S)) 
$ is generated from torsion-translated algebraic subtori of $(\bC^*)^{rs}$ via a finite sequence of taking unions, intersections, and complements.

(2) Conversely, let $S'\subset (\bC^*)^{rs}$ be a constructible subset generated from torsion-translated algebraic subtori of $(\bC^*)^{rs}$ via a finite sequence of taking unions, intersections, and complements. Then $mon^{-1}(root(S'))$ is an absolute $\bbQ$-constructible subset of $M$.
\end{lemma}

\begin{proof} (1) Let $T:=RH^{-1}(S)\cap M_{DR}(\overline{X}/D,r)^\text{lf}(\bC)$. The assumptions  guarantee that $RH(T)=S$. Let $S_1=root^{-1}(mon(S))$ and $T_1=root^{-1}(res(T))$. Then $Exp(T_1)=S_1$. 

Let $\sigma\in Gal (\bC/\bQ)$. Since $S$ is absolute, $RH(p_\sigma(T))=S^\sigma$ is $\bbQ$-constructible. Applying $\sigma$ coordinate-wise to points of $\bC^{rs}$ one obtains $\sigma(T_1)=root^{-1}(res(p_\sigma(T)))$, and therefore $Exp(\sigma(T_1))=root^{-1}(mon(S^\sigma))$ equals $\sigma(S_1)$ and is $\bbQ$-constructible.

One can identify $Exp:\bC\to \bC^*$ with the $RH$ map from the moduli of logarithmic connections on $(\bP^1, B=\{0,\infty\})$ on the trivial rank 1 line bundle to the Betti moduli of rank 1 local systems on $\bC^*$. Taking the product coordinate-wise $rs$-times corresponds to taking external product of logarithmic connections and local systems, respectively. The previous discussion implies then that  $S_1$ is an absolute $\bbQ$-constructible subset of the Betti moduli of rank 1 local systems on $(\bC^*)^{rs}$. By Theorem \ref{budurwang}, it follows that $S_1$ is obtained from finitely many torsion-translated algebraic subtori via a finite sequence of taking unions, intersections, and complements, as claimed.

(2) It is enough to prove that $S_0=mon^{-1}(root(S'))$ is absolute $\bbQ$-constructible if 
 $S'$ is a torsion-translated algebraic subtorus of $(\bC^*)^{rs}$. Let $T'=Exp^{-1}(S')$. Then $T'$ consists of all $\bZ^{rs}$-translates of a linear subvariety of $\bC^{rs}$ defined over $\bQ$. Define $T_0=RH^{-1}(S_0)\cap M_{DR}(\overline{X}/D,r)^\text{lf}(\bC)$. Then $T_0=res^{-1}(root(T'))$.
  Let $\sigma\in Gal(\bC/\bQ)$. Then $$S_0^\sigma=RH(p_\sigma(T_0)) = RH(res^{-1}(root(\sigma (T'))))= RH(res^{-1}(root(T')))=mon^{-1}(root(S')),$$
 where the maps $RH, res, mon$ are the ones defined for $(\bar X^\sigma, D^\sigma)$,  $X^\sigma$, and the first equality holds because of the assumptions (a), or (b). In particular, $S_0$ is absolute $\bbQ$-closed.
\end{proof}

Next we consider the  morphism 
 $
\det:M_B(X,r)(\bC)\to        M_B(X,1)(\bC)
$ 
defined over $\bQ$ given by the top exterior power of local systems.

\begin{lemma}\label{lemMinfa} $\;$

(1) If $S\subset M_B(X,r)(\bC)$ is an absolute $\bbQ$-constructible subset, then
$
\det(S) 
$ is absolute $\bbQ$-constructible in $M_B(X,1)(\bC)$.

(2) Conversely, let $S'\subset M_B(X,1)(\bC)$ be an absolute $\bbQ$-closed subset. Then $\det^{-1}(S')$ is an absolute $\bbQ$-closed subset of $M_B(X,r)(\bC)$.
\end{lemma}
\begin{proof} 
(1) It follows from the fact that $\det(S^\sigma)=(\det(S))^\sigma$ for all $\sigma\in Gal(\bC/\bQ)$, together with the fact that $\det(S^\sigma)$ is $\bbQ$-constructible by the assumption on $S$.

(2) Similarly, $(\det^{-1}(S'))^\sigma=\det^{-1}((S')^\sigma)$, and $\det$ is a morphism of $\bQ$-schemes.
\end{proof}

One can extend (\ref{eqDlf}) to include $det$ and yield a commutative diagram

\be\label{eqDlfb}
\begin{gathered}
\xymatrix{
M_{DR}(\bar X/D,r)^{\rm{lf}}(\bC) \ar[d]^{RH} \ar[rr]^{\det\times res\quad} & &  M_{DR}(\bar X /D,1)(\bC)\times \bC^{rs} \ar[d]^{RH\times f} & &M_{DR}(\bar X /D,1)(\bC)\times \bC^{rs} \ar[ll]_{id\times root} \ar[d]^{id\times Exp}\\
M_B(X,r)(\bC) \ar[rr]^{\det\times mon\quad\quad\quad}                       &   &   M_B(X,1)(\bC) \times ((\bC^*)^{(r)})^s  & &  M_B(X,1)(\bC) \times (\bC^*)^{rs}  \ar[ll]_{\quad id\times root} .                 
}
\end{gathered}
\ee

The proofs of the previous lemmas easily combine to give:

\begin{prop}\label{propRedCs}
Let $M$ be one of the following:

(a) $M=M_B(X,r)$, in which case we also assume that $RH:M_{DR}(\overline{X}/D,r)^\text{lf}(\bC)\to M_B(X,r)(\bC)$ is surjective (e.g. if $X$ is a curve); or,

(b) $M=M_B^s(X,r)$, in which case $\det\times mon$ will denote the restriction $(\det\times mon)|_M$.

We have:

(1) If $S\subset M$ is an absolute $\bbQ$-constructible subset, then 
$
(id\times root)^{-1}((\det\times mon)(S) )$  is generated as a subset of  $M_B(X,1)(\bC) \times (\bC^*)^{rs}$ from torsion-translated algebraic subtori via a finite sequence of taking unions, intersections, and complements.

(2) Conversely, let $S'\subset M_B(X,1)(\bC) \times (\bC^*)^{rs}$ be a torsion-translated algebraic subtorus. Then $(\det\times mon)^{-1}((id\times root)(S'))$ is an absolute $\bbQ$-closed subset of $M$.

\end{prop}

\subsection{Appendix: Corrections to \cite{BW}}\label{secApp}

 In \cite{BW} the surjection of the two Riemann-Hilbert morphisms from (\ref{eqDidr}) was taken for granted. This is still an open question. For the lack of a proof, the definition on moduli-absoluteness and its relation with absoluteness from \cite[Section 7]{BW} need to be changed in order to be correct. These changes are only necessary if the two RH maps from (\ref{eqDidr}) are not surjective.

None of the main results and none of the conjectures, all being contained in the Introduction and Sections 9-11 of \cite{BW}, need any changes. This is because they are either about absolute subsets of $M_B(X,r)$ if there are no additional conditions on $X$, or about moduli-absolute subsets in cases when both maps $RH$ are surjective, cf. Theorem \ref{thmH21}.

Corrections to \cite{BW}:

\begin{enumerate}

\item \cite[Definition 7.4.1]{BW}: in the definition of moduli-absoluteness for a subset $S$ of $R_B(X,x_0,r)(\bC)$ (respectively, of $M_B(X,r)(\bC)$), one has to assume that $S$ is contained in the image of $RH$, so that $S=S^\sigma$ if $\sigma$ is the identity. 

\item \cite[Lemma 7.4.3]{BW} should begin with: Let $T$ be a subset of the image of $RH$ in $M_B(X,r)(\bC)$ such that $q_B^{-1}(T)$ is a subset of the image of $RH$ in $R_B(X,x_0,r)(\bC)$. Then the equivalence as stated there holds.

\item \cite[Proposition 7.4.4]{BW}: $S$, $T$, and $q_B^{-1}(T)$ are assumed to be subsets of the images of the two maps $RH$.

\item \cite[Proposition 7.4.5]{BW} remains true, but the proof of (1) for $M_B$ is exactly like the proof for $R_B$, and does not follow from the result for $R_B$.

\end{enumerate}

\section{Rigid local systems}\label{secRls}

We address Conjecture A  for absolute sets of  cohomologically rigid simple local systems and give the proofs of the remaining results from Introduction. We keep the setup from the previous section. Namely, $X$ is a smooth quasi-projective complex variety, $x_0$ a base point,  $j\colon X\xhookrightarrow{} \bar{X}$ 
an open embedding into a smooth projective variety such that $D=\bar{X}-X$ is a normal crossings divisor. Let $D_i$ with $1\le i\le s$ be the irreducible components of $D$. 

\begin{defn}\label{defChrls} Let $L$ be a simple complex local system on $X$. 

(1) $L$ is {\it  rigid} if it is an isolated  point of the moduli subspace $M_B(X, r, \{C_{1},\ldots, C_s\}, E)$  of simple local systems with determinant isomorphic to $E=\det L$ and local monodromy around $D_i$ contained in the conjugacy class $C_i$ in $\GL_r(\bC)$ determined by the monodromy representation of $L$. We denote by $M_B^{rig}(X,r)(\bC)$ the subset of $M_B(X,r)(\bC)$ consisting of  rigid local systems.

(2) $L$ is {\it cohomologically rigid} if $\mathbb H^1(\bar{X}, j_{!*} End^0( L))=0$ where $End^0(L)$ denotes the local system of traceless endomorphisms of $L$ and $j_{!*}$ is the intermediate extension functor. We denote by $M_B^{cohrig}(X,r)(\bC)$ the subset of $M_B(X,r)(\bC)$ consisting of cohomologically rigid local systems.

\end{defn}

The first important properties of rigidity are:

\begin{prop}\label{propRigo}$\;$ Let $X$ be as above and $L\in M_B^s(X,r)(\bC)$. 

(1)  Each fiber of the map $\det\times mon$ from (\ref{eqDlfb}) contains at most finitely many  rigid local systems. 

(2) $M^{cohrig}_B(X,r)(\bC)$ is an absolute $\bbQ$-constructible subset of $M_B(X,r)(\bC)$.

(3)  $L$ is cohomologically rigid if and only if $L$ is a reduced isolated point of $M_B(X,r, \{C_{i}\}, E)$, with $C_i$ and $E$ as in Definition \ref{defChrls}. In particular, $
M^{cohrig}_B(X,r)(\bC)\subset M^{rig}_B(X,r)(\bC).
$

(4) If $\dim X=1$, then
$
M^{cohrig}_B(X,r)(\bC)= M^{rig}_B(X,r)(\bC).
$
Moreover, if $X$ has non-trivial rigid local systems, then $X$ is $\bP^1$ minus at least 3 points.

\end{prop}

\begin{proof} (1)
	The claim is equivalent to saying  there exist only finitely many isomorphism classes of   rigid local systems with fixed determinant whose local monodromy along the irreducible components of $D$ have fixed eigenvalues. Since $M_B(X,r, \{C_{i}\}, E)$ are $\bC$-schemes of finite type, see \cite{EG}, each $M_B(X,r, \{C_{i}\}, E)(\bC)$ contains at most finitely many  rigid local systems. Since $E$ is fixed, it is enough to show that there are only finitely many ordered sets $\{C_1,\ldots,C_s\}$ of conjugacy classes of $\GL_r(\bC)$ such that the set of eigenvalue sets $\{eig(C_1),\ldots, eig(C_s)\}\in ((\bC^*)^{r})^s$ is prescribed. This follows from the fact that there are only finitely many Jordan normal forms corresponding to a fixed set of eigenvalues with multiplicities. 
	
	(2) All the functors involved in the definition of cohomological rigidity are on the list of absolute $\bQ$-constructible functors from \cite[Theorem 6.4.3]{BW}, hence the cohomologically rigid simple local systems form an absolute $\bQ$-constructible subset of $M_B^s(X,r)(\bC)$, from which the claim follows.
	
	(3) and (4) are well-known, see for example \cite{EG} and \cite{katz96}, respectively.
\end{proof}

\begin{rmk}
It is not clear if $M_B^{rig}(X,r)(\bC)$ is even Zariski constructible. In fact it is still open if there exist rigid but not cohomologically rigid local systems.
\end{rmk}

We prove now the remaining results mentioned in Introduction.

\begin{proof}[Proof of Theorem \ref{thmChR}] (A1): It follows from the fact that A1 holds already for the moduli of simple local systems by Theorem \ref{propA1si}.

(A3): If $L$ is an absolute $\bbQ$-point of $M_B^{cohrig}(X,r)(\bC)$, then $(\det\times mon)(L)$ is the image under $id\times root$ of a torsion point, by Proposition \ref{propRedCs} (1). This is equivalent to $\det(L)$ and $C_i$ being quasi-unipotent. Conversely, if $L$ is a point of $M_B^{cohrig}(X,r)(\bC)$ and $(\det\times mon)(L)$ is the image of a torsion point under $id\times root$, then $(\det\times mon)^{-1}(\det\times mon)(L)$ is an absolute $\bbQ$-closed subset of $M_B^{cohrig}(X,r)(\bC)$ consisting of finitely many points and containing $L$, by Proposition \ref{propRedCs} (2) and Proposition \ref{propRigo} (1). Hence $L$ is also an absolute $\bbQ$-point by (1). This shows that A3 is equivalent with Simpson's conjecture in this case.

(A2): Since A1 holds for $M_B^{cohrig}(X,r)(\bC)$, the property A2 is equivalent to: let $S$ be a non-empty irreducible absolute $\bbQ$-closed subset of $M_B^{cohrig}(X,r)(\bC)$, then the subset $S'$ of absolute $\bbQ$-points of $S$ is Zariski dense in $S$. 
Let $\bar{S'}$ be  the Zariski closure of $S'$. Consider the $\bar{\bQ}$-constructible subsets in $M_B(X,1)(\bC) \times ((\bC^*)^{(r)})^s$, $$T=(\det\times mon)(S),\quad T'=(\det\times mon)(\bar{S'}).$$ By Proposition \ref{propRedCs}, $T$ is the image under $id\times root$ of a subset obtained from some torsion-translated affine algebraic subtori of $M_B(X,1)(\bC)\times (\bC^*)^{rs}$ via a finite sequence of taking finite union, intersection, complement. By the equivalence from (A3),  $T'\subset T$ contains all points in $T$ of the form $(id\times root)(\lambda)$ with $\lambda\in M_B(X,1)(\bC)\times (\bC^*)^{rs}$ torsion, hence  it is Zariski dense in  $T$. Recall that  $(\det\times mon)$ has finite fibers in $M_B^{cohrig}(X,r)(\bC)$ by Proposition \ref{propRigo} (1). Thus $\dim(\bar{S'})=\dim(T')=\dim (T)=\dim(S).$ Since $S$ is irreducible and closed,  it follows that $S=\bar{S'}$.

(A4): Let $S$ be an irreducible component of the Zariski closure  in $M_B^{cohrig}(X,r)(\bC)$  of an infinite subset of absolute
  		$\bar{\mathbb Q}$-points. By Proposition \ref{propRigo} (2), $M_B^{cohrig}(X,r)(\bC)$ is defined over $\bbQ$, hence $S$ is also defined over $\bbQ$. Then $T= (\det\times mon)(S)$ is $\bar{\bQ}$-constructible and irreducible.
  		If $S$ contains only finitely many absolute $\bar{\bQ}$-points, then it is $0$-dimensional and hence absolute $\bar{\bQ}$-closed. We may assume that $S$ contains infinitely many absolute $\bar{\bQ}$-points.
  		It then follows by the equivalence from (A3) that the set $(id\times root)^{-1}(\bar T)$ is the Zariski closure in $M_B(X,1)(\bC)\times (\bC^*)^{rs}$ of an infinite set of torsion points. In particular, it is an $S_r$-equivariant finite union of torsion-translated algebraic subtori by Theorem \ref{budurwang}. Thanks to Proposition \ref{propRedCs}, $(\det\times mon)^{-1}(\bar{T})$ is  absolute $\bar{\bQ}$-closed in $M_B^{cohrig}(X,r)(\bC)$.
  		Using again that $\det\times mon$ is a quasi-finite morphism, we have that $\dim((\det\times mon)^{-1}(\bar{T}))=\dim(\bar{T})=\dim(S).$
  		Thus $S$ is an irreducible component of $(\det\times mon)^{-1}(\bar{T})$, and by A1 it follows that  $S$ itself is absolute $\bar{\bQ}$-closed.
\end{proof}


\begin{rmk}  Conjecture A3 for $M_B^{cohrig}(X,r)$ would also hold for rank $r=3$ if \cite[Theorem 1.3]{LS} can be extended to the quasi-projective case, since  the integrality conjecture is proved in the quasi-projective case by \cite{EG}.
\end{rmk}

\begin{proof}[Proof of Theorem \ref{thmP3pts}]
Since the map $mon$ is a closed embedding in this case by {Example \ref{C^*}}, the claim follows from Lemma \ref{lemMinf}.
\end{proof}

At the intersection of the two cases considered in
Theorem \ref{thmCACR} we give an explicit description of the moduli of rigid local systems.  We can assume that $X=\bP^1\setminus D$ where $D=\{D_1,\ldots,D_s\}$ is a finite set of $s\ge 3$ distinct points, by Proposition \ref{propRigo} (4).

      \begin{proposition}\label{rk2}
      	      	Let $X=\bP^1\setminus D$ with $|D|=s\ge 3$.

	(1)	Let $L\in M_B^s(X,2)(\bC)$. The following are equivalent: \begin{itemize}
      		\item $L\in  M_B^{rig}(X,2)(\bC)$;
      		\item there are exactly $3$ points in $D$ around which the local monodromy of $L$ is not a nonzero scalar matrix.
      	\end{itemize}

	(2) $M_B^{rig}(X,2)(\bC)$ are the complex points of an open subscheme $M_B^{rig}(X,2)$ of $M_B(X,2)$ defined over $\bQ$ and there is a decomposition 
	  		\[   M_B^{rig}(X,2)=\bigsqcup_{i=1}^{{s(s-1)(s-2)}/{6}}M_i\]
	  		into $\binom{s}{3}$ disjoint closed subvarieties $M_i$ of $M_B^{rig}(X,2)$, such that each $M_i$ is absolute $\bQ$-closed in $M_B^{rig}(X,2)$ and the morphism $mon$ from (\ref{eqDlf}) induces an isomorphism from $M_i$ to the complement   of
	  		\[     V\left(\prod_{j,k,l,m=1,2}(x_{1j}x_{2k}x_{3l}x_{4m}x_{5m}\dots x_{sm}-1)\right)\] in  $	V(x_{11}x_{12}x_{21}x_{22}\dots x_{s1}x_{s2}-1, x_{41}-x_{42}, \dots, x_{s1}-x_{s2})$ 
			as subvarieties of $\left[\mathbb G_m^{(2)}\right]^{\times s}.$

(3) 	  If $s=3$ one has $M_B^{rig}(\bP^1\setminus\{0,1,\infty\},2)=M_B^s(\bP^1\setminus\{0,1,\infty\},2)=M_1.$

      \end{proposition}
    
      \begin{proof}
      	(1) Let $\rho:\pi_1(X,x_0)\ra\GL_2(\bC)$ represent $L$.
      	We calculate the dimension of the centralizer $Z_i$ of   the local monodromy $\rho(\gamma_i)$ of $L$ around each $D_i$.       	The possible Jordan canonical forms of a rank $2$ invertible  matrix are
      	\[J_1=\begin{pmatrix} \alpha&0\\0&\beta \end{pmatrix}, \;\;J_2=\begin{pmatrix} \alpha&0\\0&\alpha \end{pmatrix},\;\;J_3=\begin{pmatrix} \alpha&1\\0&\alpha \end{pmatrix},\]
      	where 
      	$ \alpha, \beta\in \mathbb C^*,\; \alpha\neq \beta.$
      	Then the dimensions of their centralizers are 2, 4, 2, respectively.

	Let  $N_0$  be the cardinality of the set
      	$\{i \in \{1,\ldots, s\}\mid \rho(\gamma_i)=\alpha I,\;\text{for some}
      	\; \alpha\in \mathbb C^*\}.$
      	Then
      	$\sum_{i=1}^{s} \dim Z_i= 4N_0+2(s-N_0)=2(N_0+s).$
      	Moreover  $L$ is rigid if and only if  
      	$2(N_0+s)=4(s-2)+2$
	by Proposition  \ref{propCz}. 	Equivalently, $N_0=s-3$.
      	
     (3) 	For the first equality, by (1) it remains to show that the local monodromy matrices $\rho(\gamma_i)$  are not scalar matrices for any simple representation $\rho$ and all $i=1, 2, 3$. Suppose to the contrary that one of the three matrices, say $\rho(\gamma_1)$, is a scalar matrix, equal to $\alpha I$ for some $\alpha\in\bC^*$. Then $ \rho(\gamma_3)= \alpha^{-1}\rho(\gamma_2)^{-1}.$
      	 It then follows that all eigenvectors of $\rho(\gamma_2)$ are eigenvectors of $\rho(\gamma_1)$ and $\rho(\gamma_3)$ as well.
      	This implies that $\rho$ is not a simple representation, which is a contradiction.   The second equality follows already from  {Example \ref{C^*}}.
	
	(2) We regard $M_B(X,2)$ as the closed subvariety $V(X_1\dots X_{s}-I)$
	 of $\GL_2(\bC)^{\times s}\sslash \GL_2(\bC)$ as in {Example \ref{exClCurve}}.
      	Let $V\subset M_B(X,2)(\bC)$ be the subset defined by $$V=\{ [(g_1, \dots, g_{s})]\in M_B(X,2)(\bC)\mid g_4=\alpha_4I, \dots, g_{s}=\alpha_{s}I, (\alpha_4, \dots, \alpha_{s})\in (\bC^*)^{s-3}\}.$$ 
Note that the set of scalar matrices in $\GL_2(\bC)$   is $\GL_2(\bC)$-invariant  Zariski closed and defined over $\bQ$. It follows that  $V$ is also an algebraic closed subset defined over $\bQ$, because it is the GIT quotient of a $\GL_2(\bC)$-invariant algebraic closed subset of $\GL_2(\bC)^{\times s}$ defined over $\bQ$.


		Let $V^s= V\cap M^s_B(X,2)(\bC)$, then we claim that $V^s=V\cap M^{rig}_B(X,2)(\bC)$. This is similar to the proof of (3).   To this end, let $L$ be a simple local system corresponding to a point $[(g_1, \dots, g_{s})]\in V^s$. If $g_1, \dots, g_s$ are nonzero scalar matrices, then $L$ would not be simple. If exactly one of $g_1, g_2, g_3$ is not a nonzero scalar matrices, say $g_1$, then $g_1=(g_2\dots g_s)^{-1}$ is a nonzero scalar matrix, a contradiction. If exactly two of $g_1, g_2, g_3$ are not nonzero scalar matrices,  say $g_1$ and $g_2$, then $g_1g_2=(g_3\dots g_s)^{-1}=\alpha I$
     	for some $\alpha\in \bC^*$. Then $\alpha^{-1}g_1g_2=I$. Thus $\alpha^{-1}g_1$ and $g_2$ have common nonzero eigenvectors, and these are then common eigenvectors for all $g_i$. In particular, $L$ is not simple in this case, a contradiction again. Thus part (1) implies the claim.

		We set $M_1=V^s$. One defines the other $M_i$ by redefining $V$ to choose other $s-3$ points among $D_1,\ldots, D_s$ for which the local monodromy should be scalar. All $M_i$ are mutually isomorphic, mutually disjoint, and cover $M^{rig}_B(X,2)(\bC)$.

		Next we show  $M_i$ are absolute $\bQ$-closed in $M_B^{rig}(X,2)(\bC)$. It is enough to show that $V$ is absolute $\bQ$-closed in $M_B(X,2)(\bC)$. Let $L\in V$. Then the Deligne canonical extension gives a logarithmic connection $(E, \nabla)$ in $RH^{-1}(L)$ whose residue along $D_j$  is $a_j\cdot I$, where $a_j\in \bC$ satisfies that  $e^{2\pi ia_j}=\alpha_j$ and the real part of $a_j$ lies in $[0,1)$, for   $j=4,\dots, s$.
		It is easy to see that $RH\circ p_\sigma (E,\nabla)$ is the isomorphism class of a local system whose local monodromy around $D_j^\sigma$ is  $e^{2\pi i\sigma(a_j)}\cdot I $ for $j=4,\dots, s$.   This shows that for every $\sigma\in \text{Gal}(\bC/\bQ)$, the set $V^\sigma$, which equals $RH\circ p_\sigma\circ RH(V)$ by Lemma \ref{lemSame}, has a similar description to $V$:
						$$V^\sigma=\{ [(g_1, \dots, g_{s})]\in M_B(X^\sigma,2)(\bC)\mid g_4=\alpha'_4I, \dots, g_{s}=\alpha'_{s}I, (\alpha'_4, \dots, \alpha'_{s})\in (\bC^*)^{s-3}\}. $$
Hence $V$ is absolute $\bQ$-closed in $M_B(X,2)(\bC)$.

		 For $g\in \GL_2(\bC)$ and $(g_1, g_2, g_3, \alpha_4I, \dots, \alpha_{s}I)\in \GL_2(\bC)^{\times s}$, the action of  $\GL_2(\bC)$ yields $g\cdot (g_1, g_2, g_3, \alpha_4I, \dots, \alpha_{s}I)= (gg_1g^{-1}, gg_2g^{-1}, gg_3g^{-1}, \alpha_4I, \dots, \alpha_{s}I).$
		Hence there is an isomorphism $V\simeq \left(\text{GL}_2(\bC)^{\times 2}\sslash \text{GL}_2(\bC)\right)\times (\bC^*)^{s-3}$ sending a point $[(g_1, g_2, g_3, \alpha_4I, \dots, \alpha_{s}I)]$ to $([ (g_2, g_3)], \alpha_4, \dots, \alpha_{s}).$ 
		Note that $g_1=(\alpha_4\dots \alpha_{s}g_2g_3)^{-1}, g_2, g_3, \alpha_4I, \dots, \alpha_{s}I$ have a common nonzero eigenvector if and only if $g_2, g_3$ have a common nonzero eigenvector, so above isomorphism identifies the subspaces of non-simple points: $$V\setminus M_1= V\setminus V^{s}=\{([ (g_2, g_3)], \alpha_4, \dots, \alpha_{s})\in V\mid g_2, g_3 \text{ have a common nonzero eigenvector} \}.$$
		Recall that   by {Example  \ref{C^*}}, \[\text{GL}_2(\mathbb C)^{\times 2}\sslash\text{GL}_2(\mathbb C)\xrightarrow{\sim} V(x_{11}x_{12}x_{21}x_{22}x_{31}x_{32}-1)\subset ((\mathbb C^*)^{(2)})^3_{(x_{11}, x_{12}, x_{21}, x_{22}, x_{31}, x_{32})}.\] 
		Hence, $$V\simeq V(x_{11}x_{12}x_{21}x_{22}x_{31}x_{32}-1)\times (\mathbb C^*)^{s-3}, $$ $$[(g_1, g_2, g_3, \alpha_4I, \dots, \alpha_{s}I)] \mapsto (eig((g_2g_3)^{-1}), eig(g_2), eig(g_3),  \alpha_4, \dots, \alpha_{s}).$$ 
		As $ t\mapsto (t,t)$ gives a closed embedding $\bC^*\simeq V(x-y)\subset (\bC^*)^{(2)}$, above isomorphism is equivalent to  $$V\simeq V(x_{11}x_{12}x_{21}x_{22}x_{31}x_{32}-1, x_{41}-x_{42}, \dots, x_{s 1}-x_{s2})\subset ((\mathbb C^*)^{(2)})^s_{(x_{11}, x_{12}, \dots, x_{s 1}, x_{s2})}, $$ $$[(g_1,\dots, g_{s})]\mapsto (eig((g_2g_3)^{-1}), eig(g_2), eig(g_3), eig(g_4), \dots, eig(g_{s})).$$
		Since  $(g_4\dots g_{s})^{-1}=(\alpha_4\dots \alpha_{s})^{-1}I$ for any $[(g_1, \dots, g_{s})]\in V$, we have $eig((g_2\dots g_m)^{-1})=(\alpha_4\dots \alpha_{s})^{-1}\cdot eig((g_2g_3)^{-1}).$
		Therefore,   the previous isomorphism is given by  $$mon \colon V\to V(x_{11}x_{12}x_{21}x_{22}\dots x_{s1}x_{s2}-1, x_{41}-x_{42}, \dots, x_{s1}-x_{s2})\subset ((\bC^*)^{(2)})^s,$$ $$[(g_1,\dots, g_{s})]\mapsto (eig(g_1)=eig((g_2\dots g_{s})^{-1}), eig(g_2), \dots, eig(g_{s})).$$
		Moreover, similarly as in {Example \ref{C^*}}, the subspace of non-simple points is identified via $mon$ with $V\setminus M_1=V\left(\prod_{j,k,
			l,m=1,2}(x_{1j}x_{2k}x_{3l}x_{4m}x_{5m}\dots x_{sm}-1)\right)\cap V.$
			Since $M_B(X,2)$ is reduced and irreducible, it can be easily seen that $V$ and $M_1$, and hence all $M_i$, are also reduced and irreducible. 
      \end{proof}

In the above proof we used:

\begin{prop}\label{propCz} (\cite[Theorem 1.1.2]{katz96}) Let $L$ be a rank $r$ simple local system on $X=\bP^1\setminus D$ where $D=\{D_1,\ldots,D_s\}$ is a finite set of $s\ge 3$ distinct points.    	   Let $Z_i$ be the centralizer  in ${Mat}_r(\mathbb C)$ of the local monodromy of $L$ around $D_i$, and regard $Z_i$ as a complex linear subspace. Then $L$ is rigid if and only if  $ \sum_{i=1}^{s} \dim Z_i= (s-2)r^2+2. $
\end{prop}

\begin{rmk} An explicit description of $M_B^{rig}(\bP^1\setminus\{0,1,\infty\},3)$  is given in \cite{HW} based on Proposition \ref{propCz}.
\end{rmk}

     \bibliographystyle{alpha}

\begin{thebibliography}{[BGLW]}
      	
      	\bibitem[BBD]{BBD} A.A. Beilinson, J. Bernstein,  P. Deligne, \textit{Faisceaux pervers.} {Analysis and topology on singular spaces, I (Luminy, 1981),} 5-171, Ast\'erisque, 100, Soc. Math. France, Paris, 1982. 
	

           
         	
      	\bibitem[BGLW]{B-all}N. Budur, N. Gatti, Y. Liu, B. Wang, \textit{On the length of perverse sheaves and $\sD$-modules.}  Bull. Math. Soc. Sci. Math. Roumanie (N.S.) 62 (110) (2019) 355-369.
	
	  	
      	\bibitem[BW]{BW}
      	N. Budur, B. Wang, \textit{Absolute sets and the Decomposition Theorem.} Ann. Sci. \'Ecole Norm. Sup. 53 (2020) 469-536 .     
	      	
      	\bibitem[CS]{CS} K. Corlette, C. Simpson, \textit{On the classification of rank-two representations of quasiprojective fundamental groups.} 	Compos. Math. 144 (2008) 1271-1331. 
      	
      	\bibitem[D]{deligne06}
      	P. Deligne, \textit{\'Equations diff\'erentielles \`a points singuliers r\'egulies.}
      	Lecture Notes in Mathematics 163, Springer-Verlag, 1970. iii+133 pp.
     	
\bibitem[De]{De} J.P. Demailly, {\it Complex Analytic and Differential Geometry.} Universit\'e de Grenoble I, 1997. 518 pp.
	 
	
\bibitem[EG]{EG} H. Esnault, M. Groechenig, \textit{Cohomologically rigid local systems and integrality.} Selecta Math. 24 (2018) 4279-4292.
	
      	
      	\bibitem[EK1]{EK} H. Esnault, M. Kerz, {\it Arithmetic subspaces of moduli spaces of rank one local systems.}  Camb. J. Math. 8 (2020) 453-478.
	
	\bibitem[EK2]{EK2} H. Esnault, M. Kerz, {\it Density of arithmetic representations of function fields.} arXiv:2005.12819.
	
	\bibitem[EV1]{EV} H. Esnault, E. Viehweg,  {\it Logarithmic de Rham complexes and vanishing theorems.}  Invent. Math.  86  (1986) 161-194.
	
	\bibitem[EV2]{EVa} H. Esnault, E. Viehweg, {\it Semistable bundles on curves and irreducible representations of the fundamental group.} Algebraic Geometry, Hirzebruch 70 (Warsaw, 1998), 129-138, Contemp. Math. 241, Amer. Math. Soc., 1999. 
        
	
	\bibitem[HTT]{H} R. Hotta, K. Takeuchi,  T. Tanisaki, {\it $\sD$-modules, perverse sheaves, and representation theory.}  Progress in Mathematics 236, Birkh\"auser, 2008. xii+407 pp.
      	
           	
      	\bibitem[K]{katz96}
      	N.M. Katz, \textit{Rigid local systems.} Annals of Mathematics Studies 139, Princeton University Press, 1996. viii+223 pp.	
      	
      	\bibitem[LS]{LS}
      	A. Langer, C. Simpson, \textit{Rank 3 rigid representations of projective fundamental groups.}
      	Compos. Math. 154 (2018) 1534-1570.

\bibitem[L]{L}	L.A. Lerer, {\it Cohomology jump loci and absolute sets for singular varieties.} arXiv:2011.14184. 
      	
\bibitem[LM]{lubotzky1985varieties}
      	A. Lubotzky, A. Magid, \textit{Varieties of representations of finitely generated groups.}  Mem. Amer. Math. Soc. 58, 1985. xi+117 pp.      	
      	
      	\bibitem[N1]{N}
      	N. Nitsure, \textit{ Moduli of semistable logarithmic connections.}
      	J. Amer. Math. Soc. 6 (1993) 597-609.

\bibitem[N2]{N2} N. Nitsure, {\it Moduli of regular holonomic $\sD$-modules with normal crossing singularities.}
Duke Math. J. 99 (1999) 1-39.
		

\bibitem[NS]{NS} N. Nitsure, C. Sabbah, {\it Moduli of pre-$\sD$-modules, perverse sheaves and the Riemann-Hilbert morphism, I,} Math. Ann. 306 (1996), 47-73.
	
\bibitem[O]{O} M. Ohtsuki, {\it A residue formula for Chern classes associated with logarithmic connections.} Tokyo J. Math. 5 (1982) 13-21.
         	
\bibitem[S1]{S93} C. Simpson, {\it Subspaces of moduli spaces of rank one
local systems.}  Ann. Sci. \'{E}cole Norm. Sup. (4)  26  (1993) 361-401.  
       	
      	\bibitem[S2]{S1}
      	C. Simpson, \textit{Moduli of representations of the fundamental group of a smooth projective variety {I}.} Inst. Hautes \'Etudes Sci. Publ. Math. 79 (1994) 47-129.  
      	
      	\bibitem[S3]{S2}
      	C. Simpson, \textit{Moduli of representations of the fundamental group of a smooth projective variety {II}.} Inst. Hautes \'Etudes Sci. Publ. Math. 80 (1994) 5-79.  
      	
          	
 \bibitem[W]{HW}    H. Wang, {\it Absolute sets of local systems}. PhD thesis, KU Leuven, 2021.
      	
      	   	
      \end{thebibliography}

  \end{document}